\documentclass[11pt]{article}

\usepackage{fullpage}
\usepackage{amsmath}
\usepackage{amssymb}
\usepackage{amsthm}
\usepackage{fancyhdr}
\usepackage{algorithm}
\usepackage{algorithmic}
\usepackage[usenames]{color}
\usepackage{multirow}
\usepackage{graphicx}
\usepackage{tikz}
\usepackage{bm}
\usepackage[caption=false,font=footnotesize]{subfig}
\allowdisplaybreaks

\newtheorem{theorem}{Theorem}[section]
\newtheorem{lemma}[theorem]{Lemma}

\theoremstyle{plain} 
\newcommand{\thistheoremname}{}
\newtheorem*{genericthm*}{\thistheoremname}
\newenvironment{namedthm*}[1]
  {\renewcommand{\thistheoremname}{#1}%
   \begin{genericthm*}}
  {\end{genericthm*}}
\def\la{\left\langle}
\def\ra{\right\rangle}
\def\ln{\left\|}
\def\rn{\right\|}
\def\lb{\left(}
\def\rb{\right)}
\def\lsb{\left[}
\def\rsb{\right]}
\def\lcb{\left\{}
\def\rcb{\right\}}
\def\lab{\left|} 
\def\rab{\right|}
\def\Pom{\frac{8}{3}\beta\log(n)}

\def\tilm{\widehat{m}}
\def\subto{\mbox{ subject to }}
\def\A{\mathcal{A}}

\def\R{\mathbb{R}}
\def\P{\mathcal{P}}
\def\Prob{\mathbb{P}}

\def\E{\mathbb{E}}

\def\I{\mathcal{I}}
\def\H{\mathcal{H}}
\def\T{\mathcal{T}}
\def\hT{\widehat{T}}
\def\hZ{\widehat{Z}}
\def\hU{\widehat{U}}
\def\hV{\widehat{V}}
\def\rank{\mbox{rank}}
\def\trim{\mbox{\tt trim}}
\def\vep{\varepsilon}
\newcommand\numberthis{\addtocounter{equation}{1}\tag{\theequation}}

\newcommand{\sigmin}[1]{\sigma_{\min}\lb #1\rb}
\newcommand{\sigmax}[1]{\sigma_{\max}\lb #1\rb}

\newcommand{\dobr}[1]{\lb #1\rb}
\newcommand{\donr}[1]{\ln #1\rn}

\makeatletter
\newcommand*{\rom}[1]{\expandafter\@slowromancap\romannumeral #1@}
\makeatother
\title{Guarantees of Riemannian Optimization for Low Rank  Matrix Completion}
\author{Ke Wei\thanks{Department of Mathematics, University of California at Davis, kewei@math.ucdavis.edu}\quad Jian-Feng Cai\thanks{Department of Mathematics, Hong Kong University of Science and Technology, \{jfcai,masyleung\}@ust.hk}\quad Tony F. Chan\thanks{Office of the President, Hong Kong University of Science and Technology, tonyfchan@ust.hk}\quad Shingyu Leung\footnotemark[2]}

\begin{document}
\maketitle
\begin{abstract}
We study the Riemannian optimization methods on the embedded manifold of low rank matrices for the problem of matrix completion, which is about recovering a low rank matrix from its partial entries.
Assume $m$ entries of an $n\times n$ rank $r$ matrix are sampled independently and uniformly with replacement. We first prove that with high probability the Riemannian gradient descent and conjugate gradient descent algorithms initialized by one step hard thresholding are guaranteed to converge linearly to the measured matrix provided
\begin{align*}
m\geq C_\kappa n^{1.5}r\log^{1.5}(n),
\end{align*}
where $C_\kappa$ is a numerical constant depending on the condition number of the underlying matrix. 
The sampling complexity has been further improved to 
\begin{align*}
m\geq C_\kappa nr^2\log^{2}(n)
\end{align*}
via  the resampled Riemannian gradient descent initialization. The analysis of the new initialization
procedure relies on an asymmetric restricted
isometry property of the sampling operator and the curvature  of the low rank 
matrix manifold. Numerical simulation shows that the algorithms are able to recover a low rank matrix
from nearly the minimum number of measurements.
\end{abstract}



\section{Introduction}\label{sec:introduction}
The problem of matrix completion attempts to recover a low rank matrix from a subset of sampled entries. 
This problem arises from  a wide variety of practical context, such as model reduction \cite{LiVa2009interior}, pattern recognition \cite{Eld2007pattern}, and machine learning \cite{AmFiSrUl2007class,ArEvPo2007mult}.  
 Since the  
data matrix is assumed to be low rank, it is natural to seek the lowest rank matrix consistent with the observed entries by solving a 
rank minimization problem 
\begin{align*}
\min_{Z\in\R^{n\times n}}\rank(Z)\subto\P_\Omega(Z) = \P_\Omega(X),\numberthis\label{eq:rank_min}
\end{align*}
where $X\in\R^{n\times n}$ is the underlying matrix to be reconstructed, $\Omega$ is a subset of indices for the known entries, and 
$\P_\Omega$ is the associated sampling operator which acquires only the entries indexed by $\Omega$. Here we restrict our discussion to square matrices for ease of exposition, but emphasize that all the results can be extended straightforwardly  to the case where the measured matrix is rectangular. Problem \eqref{eq:rank_min} is generally NP-hard \cite{hky2006mccomplexity} and computationally intractable.  
In a seminal paper of Cand\`es and Recht \cite{candesrecht2009mc}, the authors studied nuclear norm minimization for matrix completion, where 
 the rank objective in \eqref{eq:rank_min}  is replaced by the nuclear norm of matrices which is the sum of the
singular values. 
 It was observed in the same paper  that we cannot  expect to recover all the low rank matrices from their partial  known entries. For example, if there exists only one nonzero entry in a matrix, we need to see almost all of its entries to infer that it is not a zero matrix.  This observation motivates the following assumption.
\begin{namedthm*}{A0}[\cite{candesrecht2009mc}]\label{assump_mu0}
Let $X\in\R^{n\times n}$ be a rank $r$ matrix with the reduced singular value decomposition (SVD) $X=U\Sigma V^*$. We assume  $X$ is $\mu_0$-incoherent; that is, there exists an  absolute numerical constant $\mu_0>0$ such that 
\begin{align*}
\ln\P_U\lb e_i\rb\rn\leq\sqrt{\frac{\mu_0 r}{n}}\quad\mbox{and}\quad\ln\P_V\lb e_j\rb\rn\leq\sqrt{\frac{\mu_0 r}{n}}\numberthis\label{eq:def_mu0}
\end{align*}
for $1\leq i,j\leq n$. Here $e_\ell$ $(\ell=i,j)$ is the $\ell$-th  canonical basis of $\mathbb{R}^n$, and $\P_U$ and $\P_V$ are the orthogonal projections onto the column and row spaces of $X$ respectively.
\end{namedthm*}
A matrix is $\mu_0$-incoherent implies  that its singular vectors  are weakly correlated with the canonical basis. In addition,  \eqref{eq:def_mu0} is equivalent to 
$\ln U^{(i)}\rn\leq\sqrt{\mu_0 r/n}$ and $\ln V^{(j)}\rn\leq\sqrt{\mu_0 r/n}$,
where $U^{(i)}$ and $V^{(j)}$ are the $i$-th and $j$-th rows of $U$ and $V$ respectively. 
We also require that the measured matrix cannot be too spiky.
\begin{namedthm*}{A1}[\cite{rhkamso}]
Let $X\in\R^{n\times n}$ be a rank $r$ matrix. We assume that exists another absolute numerical constant $\mu_1$ such that 
\begin{align*} 
\ln X\rn_\infty\leq\mu_1\sqrt{\frac{r}{n^2}}\ln X\rn.\numberthis\label{eq:def_mu1}
\end{align*}
\end{namedthm*}
This manuscript investigates recovery guarantees of a class of Riemannian gradient descent and conjugate gradient algorithms for matrix completion under the sampling with replacement model.
{\it We first establish the local convergence of the Riemannian gradient decent algorithm and a restarted variant of the Riemannian conjugate gradient
descent algorithm. Then we prove that $O(nr^2\log^2(n))$ number of measurements are sufficient for the Riemannian optimization algorithms to converge linearly to the underlying low rank matrix when the algorithms are properly initialized.}
\subsection{Notations and Organization of the Manuscript}
The rest of the manuscript is organized as follows. We first summarize the notations used throughout this manuscript in the remainder of this section. In Sec.~\ref{sec:algs}, we present the Riemannian gradient descent and conjugate gradient descent algorithms based on the embedded manifold of low rank matrices from the perspective of iterative hard thresholding algorithms.  
Then we present the theoretical guarantees of the algorithms.
Empirical observations in Sec.~\ref{sec:numerics} demonstrates the efficiency and robustness of the Riemannian optimization algorithms. The proofs of the main results are presented in Sec.~\ref{sec:proofs} and Sec.~\ref{sec:conclusion} concludes this manuscript with potential future directions. Finally, the Appendix provides proofs for the supporting 
technical lemmas.

Throughout this manuscript, we denote matrices by uppercase letters and vectors by lowercase letters.  In particular, $X$ denotes the $n\times n$ rank $r$ matrix to be reconstructed, with  $\sigmin{X}$  and  $\sigmax{X}$ being its smallest nonzero and largest singular values repectively. The condition number of $X$ is denoted by $\kappa$, which is defined as $\kappa=\frac{\sigmax{X}}{\sigmin{X}}$. We always assume that $X$ obeys the conditions set out in \textbf{A0} and \textbf{A1}. 
For any matrix $Z$, the spectral norm of $Z$ is denoted by $\ln Z\rn$, the Frobenius norm is denoted by $\ln Z\rn_F$, and the maximum magnitude of its entries is denoted by $\ln Z\rn_\infty$. We use $Z^{(i)}$ to represent the $i$-th row of $Z$. The  Euclidean norm of a vector $x$ is denoted by $\ln x\rn$. Operators that map matrices  to matrices are denoted by calligraphic letters. In particular, $\I$ denotes the identity operator. The spectral norm of a linear operator $\mathcal{A}$ is denoted by $\ln\A\rn$.

Given a collection of indices $\Omega$ with $|\Omega|=m$ (counting multiplicity) for the observed entries, we define $p=\frac{m}{n^2}$ which is the probability of each entry being observed, and $\P_{\Omega}$ represents the sampling operator that maps a matrix to its component-wise product with the matrix $\sum_{(i,j)\in\Omega}e_ie_j^*$. In the main results, $\vep_0$ is a positive numerical constant which controls the size of the contraction neighborhood of an algorithm and its value only depends on  the algorithm. Finally, we use $C$ to denote an absolute numerical constant whose value may change according to context.

\section{Algorithms and Main Results}\label{sec:algs}
We consider the {\it sampling with replacement model } for $\Omega$ in which each index  is sampled independently from the uniform distribution on $\{1,\cdots,n\}\times \{1,\cdots,n\}$.  This model can be viewed as a proxy for the uniform sampling model and {the failure probability under the uniform sampling model is less than or equal to the failure probability under the sampling with replacement model \cite{recht2011simple}.}
Though there may exist duplicates in $\Omega$, the maximum number of duplication can be upper bounded with high probability. 
\begin{lemma}[\cite{recht2011simple}]\label{lem:repetition}
With probability at least $1-n^{2-2\beta}$, the maximum number of repetitions of any entry in $\Omega$ is less than 
$\frac{8}{3}\beta\log(n)$ for $n\geq 9$ and $\beta>1$.
\end{lemma}It follows immediately from Lem.~\ref{lem:repetition} that with high probability we have $\ln\P_\Omega\rn\leq \frac{8}{3}\beta\log(n)$.

Many computationally efficient  algorithms have been designed to target the matrix completion problem directly by considering the following non-convex formulation\footnote{Note that the objective function in \eqref{eq:obj_func} is the same as $\frac{1}{2}\sum_{(i,j)\in\Omega}(Z_{ij}-X_{ij})^2$, but it might not be equal to $\frac{1}{2}\ln \P_\Omega\lb Z-X\rb\rn_F^2$ because in general  $\P_\Omega$ is not a projection (i.e., $\P_\Omega^2\ne\P_\Omega$) when $\Omega$ is sampled with replacement.}
\begin{align*}
\min_{Z\in\R^{n\times n}} \frac{1}{2}\la Z-X,\P_\Omega\lb Z-X\rb\ra\subto \rank(Z)=r,\numberthis\label{eq:obj_func}
\end{align*}
 see \cite{CGIHT, jmd2010svp,cevher2012matrixrecipes,tw2012nihtmc,haldahernando2009pf, wyz2012lmafit, tannerwei_asd,bart2012riemannian,ngosaad2012scgrassmc,mas2012newgeometry,mmbs2012fixrank,JaNe2015fast,boumalabsil2011rtrmc} and references therein for an incomplete review. Normalized iterative hard thresholding \cite{tw2012nihtmc} (NIHT, also known as SVP \cite{jmd2010svp} or IHT \cite{goldfarbma2011fpca} when the stepsize is fixed) is a projected gradient descent algorithm which updates the estimate along the  gradient descent direction of the objective function in \eqref{eq:obj_func}, 
followed by the projection onto the set of rank $r$ matrices via the hard thresholding operator: 
\begin{align*}
X_{l+1} = \H_r\lb X_l+\alpha_l\P_\Omega\lb X-X_l\rb\rb,\numberthis\label{eq:niht}
\end{align*}
where the hard thresholding operator $\H_r$  first computes the SVD  of a matrix and then  sets all but the $r$ largest singular values to zero
\begin{equation}
\H_r(Z):=P\Lambda_r Q^*\quad\mbox{where $Z=P\Lambda Q^*$ is the SVD of $Z$, and } \Lambda_r(i,i):=\begin{cases}
\Lambda(i,i) & i\leq r\\
0 & i > r.
\end{cases}
\end{equation}

NIHT suffers from the slow asymptotic convergence rate of first order methods as a gradient descent algorithm. To overcome this slow convergence rate of NIHT, a  conjugate gradient iterative hard thresholding (CGIHT) algorithm has been developed in \cite{CGIHT}. In each iteration of CGIHT, the new search direction is computed as a weighted sum of the gradient descent direction and the past search direction, and the selection of the weights guarantees that the new search direction is conjugate orthogonal to the past search direction when projected onto the column space of the current estimate.
In both NIHT and CGIHT, we need to compute the largest $r$ singular values and their associated singular vectors of an $n\times n$ matrix in each iteration. Although Krylov subspace methods can be applied to compute this truncated SVD, it is still computationally expensive especially when $n$ is large and $r$ is moderate. The Riemannian optimization algorithms based on the embedded manifold of rank $r$ matrices can reduce the computational cost
of the truncated SVD dramatically by exploring the local structure of low rank matrices.

We first present the Riemannian gradient descent method for matrix completion in Alg.~\ref{alg:rgrad}. Let $X_l=U_l\Sigma_lV_l^*$ be the current estimate and $T_l$ be the tangent space of the rank $r$ matrix manifold at $X_l$; that is,
 \begin{align*}
 T_l=\lcb U_lZ_1^*+Z_2V_l^*:~Z_1,Z_2\in\R^{n\times r}\rcb.
 \end{align*}
 The new estimate $X_{l+1}$ is obtained by updating $X_l$ along $\P_{T_l}\lb G_l\rb$, the gradient descent direction projected onto the tangent space $T_l$, using the steepest stepsize $\alpha_l$, and then followed by hard thresholding the estimate back onto the set of rank $r$ matrices.
 \begin{algorithm}[t]
\caption{Riemannian Gradient Descent (RGrad)}\label{alg:rgrad}
\begin{algorithmic}
\STATE\textbf{Initilization}: $X_0=U_0\Sigma_0V_0^*$ 
\FOR{$l=0,1,\cdots$}
\STATE 1. $G_l=\P_\Omega\lb X- X_l\rb$
\STATE 2. $\alpha_l=\frac{\ln\P_{T_l}\lb G_l\rb\rn_F^2}{\la \P_{T_l}\lb G_l\rb, \P_\Omega\P_{T_l}\lb G_l\rb\ra}$
\STATE 3. $W_l = X_l+\alpha_l\P_{T_l} \lb G_l\rb$
\STATE 4. $X_{l+1}=\H_r\lb W_l\rb$
\ENDFOR
\end{algorithmic}
\end{algorithm}

In the Riemannian conjugate gradient descent method (Alg.~\ref{alg:rcg}), the search direction is a linear combination of the projected gradient descent direction and the past search direction projected onto the tangent space of the current estimate.  The orthogonalization weight $\beta_l$ in Alg.~\ref{alg:rcg} is selected in a way such that $\P_{T_l}\lb P_l\rb$ is conjugate orthogonal to $\P_{T_l}\lb P_{l-1}\rb$, which follows from CGIHT in \cite{CGIHT}. It is worth noting that there are other selections for $\beta_l$ following from the non-linear conjugate gradient descent method in convex optimization \cite{bart2012riemannian}. In particular, the Riemannian conjugate gradient descent algorithm (LRGeomCG) developed  in \cite{bart2012riemannian} for matrix completion uses the Polak-Ribi\`ere+ selection for $\beta_l$. In this manuscript, we interpret the Riemannian optimization methods directly  as iterative hard thresholding algorithms with subspace projections. For the differential geometry ideas behind the algorithms, we refer the readers to \cite{AbMaSe2008manifold}.
\begin{algorithm}[t]
\caption{Riemannian Conjugate Gradient Descent (RCG)}\label{alg:rcg}
\begin{algorithmic}
\STATE\textbf{Initilization}: $X_0=U_0\Sigma_0V_0^*$, $\beta_0=0$ and $P_{-1}=0$
\FOR{$l=0,1,\cdots$}
\STATE 1. $G_l=\P_\Omega\lb X- X_l\rb$
\STATE 2. $\beta_l =  -\frac{\la\P_{T_l}\lb G_l\rb,\P_\Omega\P_{T_l}\lb P_{l-1}\rb\ra}{\la\P_{T_l}\lb P_{l-1}\rb, \P_\Omega\P_{T_l}\lb P_{l-1}\rb\ra}$
\STATE 3. $P_l=\P_{T_l}\lb G_l\rb+\beta_l \P_{T_l}\lb P_{l-1}\rb$
\STATE 4. $\alpha_l=\frac{\la \P_{T_l}\lb G_l\rb, \P_{T_l}\lb P_l\rb\ra}{\la \P_{T_l}\lb P_l\rb, \P_\Omega\P_{T_l}\lb P_l\rb\ra}$
\STATE 5. $W_l=X_l+\alpha_l P_l$
\STATE 6. $X_{l+1}=\H_r\lb W_l\rb$
\ENDFOR
\end{algorithmic}
\end{algorithm}

Though we still need to compute the truncated SVD of $W_l$ to find its best rank $r$ approximation, the truncated SVD can be computed efficiently using two QR factorizations of $n\times r$ matrices and one full SVD of a $2r\times 2r$ matrix as matrices in $T_l$ are at most rank $2r$. To see this, notice that $W_l$ in Algs.~\ref{alg:rgrad} and \ref{alg:rcg} can be rewritten as
\begin{align*}
W_l = X_l+\P_{T_l}\lb Z_l\rb,
\end{align*} 
where $Z_l=\alpha_lG_l$ in Alg.~\ref{alg:rgrad} and $Z_l=\alpha_l(G_l+\beta_lP_{l-1})$ in Alg.~\ref{alg:rcg}.
So 
\begin{align*}
W_l &= U_l\Sigma_lV_l^*+U_lU_l^*Z_l+Z_lV_lV_l^*-U_lU_l^*Z_lV_lV_l^*\\
&=U_l\lb \Sigma_l+U_l^*Z_lV_l\rb V_l^*+U_lU_l^*Z_l(I-V_lV_l^*)+(I-U_lU_l^*)Z_lV_lV_l^*\\
&:= U_l\lb \Sigma_l+U_l^*Z_lV_l\rb V_l^*+U_lY_1^*+Y_2V_l^*.
\end{align*}
Let $Y_1=Q_1R_1,~Y_2=Q_2R_2$ be the QR factorizations of $Y_1$ and $Y_2$ respectively. Then  $U_l^*Q_2=0$ and $V_l^*Q_1=0$, and we have 
\begin{align*}
W_l &= \begin{bmatrix}
U_l&Q_2
\end{bmatrix}
\begin{bmatrix}
 \Sigma_l+U_l^*Z_lV_l & R_1^*\\
 R_2 & 0
\end{bmatrix}
\begin{bmatrix}
V_l^*\\Q_1^*
\end{bmatrix}\\
&:=\begin{bmatrix}
U_l&Q_2
\end{bmatrix} M_l\begin{bmatrix}
V_l^*\\Q_1^*
\end{bmatrix}.
\end{align*}
Since $\begin{bmatrix}
U_l&Q_2
\end{bmatrix}$ and 
$\begin{bmatrix}
V_l & Q_1
\end{bmatrix}$
are both unitary matrices, the SVD of $W_l$ can be obtained from the SVD of $M_l$, which is  a $2r\times 2r$ matrix.
\subsection{Local Convergence}
We first identify a small neighborhood around the measured low rank matrix such that for any given initial guess in this neighborhood, Algs.~\ref{alg:rgrad} and \ref{alg:rcg} will converge linearly to the true solution. 
For the Riemannian gradient descent algorithm (Alg.~\ref{alg:rgrad}), we have the following theorem.
\begin{theorem}[Local Convergence of Riemannian Gradient Descent]\label{thm:rgrad_det}
Let $X\in\R^{n\times n}$ be the measured rank $r$ matrix and $T$ be the tangent space of the rank $r$ matrix manifold at $X$. Suppose
\begin{align*}
&\ln\P_\Omega\rn\leq\frac{8}{3}\beta\log(n),\numberthis\label{eq:rgrad_det_cond1}\\
&\ln \P_T-p^{-1}\P_T\P_\Omega\P_T\rn\leq\varepsilon_0,\numberthis\label{eq:rgrad_det_cond2}\\
&\frac{\ln X_0-X\rn_F}{\sigmin{X}}\leq\frac{3p^{1/2}\varepsilon_0}{16\beta\log(n)(1+\varepsilon_0)},\numberthis\label{eq:rgrad_det_cond3}
\end{align*}
where $\beta>1$, and $\vep_0$ is a positive numerical constant such that 
\begin{align*}
\nu_g=\frac{18\vep_0}{1-4\vep_0}<1.\numberthis\label{eq:rgrad_mu}
\end{align*} 
Then the iterates $X_l$ generated by Alg.~\ref{alg:rgrad} statisfy
\begin{align*}
\ln X_l-X\rn_F\leq \nu_g^l\ln X_0-X\rn_F.
\end{align*}
\end{theorem}

To state a similar result for the Riemannian conjugate gradient descent algorithm, we introduce a restarted 
variant of Alg.~\ref{alg:rcg}; that is, $\beta_l$ is set $0$ and restarting occurs as long as  either of the following conditions 
is violated 
\begin{align}\label{eq:restart_cond}
\begin{split}
\frac{\lab\la\P_{T_l}\lb G_l\rb, \P_{T_l}\lb P_{l-1}\rb\ra\rab}{\ln\P_{T_l}\lb G_l\rb \rn_F\ln \P_{T_l}\lb P_{l-1}\rb\rn_F}\leq \kappa_1,\quad\ln \P_{T_l}\lb G_l\rb\rn_F\leq \kappa_2\ln \P_{T_l}\lb P_{l-1}\rb\rn_F .
\end{split}
\end{align}
The restarting conditions are introduced not only for the sake of proof, but also to improve the robustness of the non-linear conjugate gradient descent methods \cite{wrightnocedal2006no}. The first restarting condition guarantees that the residual will be substantially orthogonal to the past search direction when projected onto the tangent space of current estimate so that the new search direction can be sufficiently gradient related. In the classical CG algorithm for linear systems, the residual is exactly orthogonal to all the past search directions.
Roughly speaking, the second restarting condition implies that the projection of current residual cannot be too large when compared to the projection of the past residual since the search direction is gradient related by the first restarting condition.  In our implementations, we take $\kappa_1=0.1$ and $\kappa_2=1$.

\begin{theorem}[Local Convergence of restarted Riemannian Conjugate Gradient Descent]\label{thm:rcg_det}
Let $X\in\R^{n\times n}$ be the measured rank $r$ matrix and $T$ be the tangent space of the rank $r$ matrix manifold at $X$. Suppose
\begin{align*}
&\ln\P_\Omega\rn\leq\frac{8}{3}\beta\log(n),\numberthis\label{eq:rcg_det_cond1}\\
&\ln \P_T-p^{-1}\P_T\P_\Omega\P_T\rn\leq\varepsilon_0,\numberthis\label{eq:rcg_det_cond2}\\
&
\frac{\ln X_0-X\rn_F}{\sigmin{X}}\leq\frac{3p^{1/2}\varepsilon_0}{16\beta\log(n)(1+\varepsilon_0)},\numberthis\label{eq:rcg_det_cond3}
\end{align*}
where $\beta>1$, and $\vep_0$ is a positive numerical constant. 
Assume that $\vep_0$ obeys  
\begin{align*}
\tau_1+\tau_2<1,\numberthis\label{eq:rcg_gamma}
\end{align*}
where 
\begin{align*}
\tau_1=\frac{18\varepsilon_0-10\kappa_1\varepsilon_0(1+4\varepsilon_0)}{\dobr{1-4\vep_0}-\kappa_1\dobr{1+4\vep_0}}+\frac{4\kappa_2\varepsilon_0+\kappa_1\kappa_2}{1-4\varepsilon_0},\quad\tau_2=\frac{8\kappa_2\varepsilon_0+2\kappa_1\kappa_2}{1-4\varepsilon_0}.
\end{align*}
Then we have $\nu_{cg}=\frac{1}{2}\lb\tau_1+\sqrt{\tau_1^2+4\tau_2}\rb<1$ and the iterates $X_l$ generated by Alg.~\ref{alg:rcg} subject to the restarting
conditions in \eqref{eq:restart_cond} satisfy 
\begin{align*}
\ln X_l-X\rn_F\leq\nu_{cg}^l\ln X_0-X\rn_F.
\end{align*}
When $\kappa_1=\kappa_2=0$, \eqref{eq:rcg_gamma} is reduced to \eqref{eq:rgrad_mu}.
On the other hand we have 
\begin{align*}
\lim_{\vep_0\rightarrow 0} (\tau_1+\tau_2)=3\kappa_1\kappa_2. 
\end{align*}
So if $\kappa_1\kappa_2<1/3$, $\tau_1+\tau_2$ can be less than one when $\vep_0$ is small. In particular, when $\kappa_1=0.1$
and $\kappa_2=1$, a sufficient condition for $\tau_1+\tau_2<1$ is $\varepsilon_0\leq 0.01$. 
\end{theorem}
The proofs of Thms.~\ref{thm:rgrad_det} and \ref{thm:rcg_det} are presented in Secs.~\ref{subsec:proof_rgrad_det} and \ref{subsec:proof_rcg_det} respectively. 
\subsection{Initialization and Recovery Guarantees}
In Thms.~\ref{thm:rgrad_det} and \ref{thm:rcg_det}, we list three conditions 
to guarantee the convergence of the algorithms. 
The requirement for 
 $\P_\Omega$ to be bounded in \eqref{eq:rgrad_det_cond1} and \eqref{eq:rcg_det_cond1} is just an artifact of the 
 sampling model, and it can be satisfied with probability at least 
 $1-n^{2-2\beta}$ following from Lem.~\ref{lem:repetition}.
 The second condition \eqref{eq:rgrad_det_cond2} and \eqref{eq:rcg_det_cond2}  
 is a local restricted isometry property 
 which has been established in \cite{candesrecht2009mc} for the Bernoulli  model and in \cite{gross2011recoverlowrank,recht2011simple} for the sampling with replacement model,
and it also plays a key role in 
  nuclear norm minimization for matrix completion. For the sampling with
 replacement model, Thm.~\ref{thm:isometry} implies that 
 as long as 
 \begin{align*}
 m\geq C\beta\lb\frac{\mu_0}{\vep_0^2}\rb nr\log(n),
 \end{align*}
 $\ln \P_T-p^{-1}\P_T\P_\Omega\P_T\rn\leq\varepsilon_0$ with probability at least
 $1-2n^{2-2\beta}$.
 Thus the only issue that remains to be addressed is how to produce an initial guess that is sufficiently close to the measured matrix. We will consider two initialization 
 strategies.
 \subsubsection{Initialization via One Step Hard Thresholding} 
 A widely used initialization  for matrix completion is to set $X_0=\H_r\lb p^{-1}\P_\Omega(X)\rb$. 
 The approximation error $\ln X_0-X\rn_F$ can be estimated as 
 follows. 
 \begin{lemma}\label{lem:error_of_init1}
 Suppose $\Omega$ with $|\Omega|=m$ is a set of indices sampled independently and uniformly with replacement.
 Let $X_0=\H_r\lb p^{-1}\P_\Omega(X)\rb$. Then for all $\beta >1$
 \begin{align*}
 \ln X_0-X\rn_F\leq \sqrt{\frac{64\beta\mu_1^2nr^2\log(n)}{3m}}\ln X\rn
 \end{align*}
 with probability at least $1-2n^{1-\beta}$ provided $m\geq 6\beta n \log(n)$.
 \end{lemma}
The proof of Lem.~\ref{lem:error_of_init1} is presented in Sec.~\ref{subsec:proof_err_init1}, which follows from a direct application 
of Thm.~\ref{thm:operator_norm} for the sampling with replacement model. A similar inequality has been established in \cite{rhkamso} for $\widehat{X_0}=\H_r\lb p^{-1}\widehat{\P_\Omega(X)}\rb$ under the Bernoulli model using random graph theory, where $\widehat{\P_\Omega(X)}$ is a trimmed matrix obtained by setting all the rows and columns of $\P_\Omega(X)$ with too many observed entries to zero.  Moreover, an application of  the  standard Chernoff bound can further show that with high probability  there are no rows or columns of $\P_\Omega(X)$ with too many entries  \cite{JaNeSa2012ammc}.

It follows from Lem.~\ref{lem:error_of_init1} that the third condition \eqref{eq:rgrad_det_cond3} and \eqref{eq:rcg_det_cond3}  in Thms.~\ref{thm:rgrad_det} and \ref{thm:rcg_det}  can be satisfied with probability at least $1-2n^{1-\beta}$ if 
\begin{align*}
m\geq C\beta^{3/2}\lb\frac{\mu_1(1+\vep_0)}{\vep_0}\rb\kappa n^{3/2}r\log^{3/2}(n).
\end{align*}
Therefore we can establish the following theorem.
\begin{theorem}[Recovery Guarantee \rom{1}]\label{thm:main_init1}
Let $X\in\R^{n\times n}$ be the measured rank $r$ matrix. Suppose $\Omega$ with $|\Omega|=m$ is a set of indices sampled independently and uniformly with replacement. Let $X_0=\H_r\lb p^{-1}\P_\Omega(X)\rb$. Then for all $\beta>1$, the iterates of 
the Riemannian gradient descent algorithm (Alg.~\ref{alg:rgrad}) and the restarted Riemannian conjugate gradient algorithm (Alg.~\ref{alg:rcg}, restarting when either the inequality in \eqref{eq:restart_cond} is violated) with $\kappa_1\kappa_2<1/3$ is guaranteed to converge to $X$ with probability at least $1-3n^{2-2\beta}-2n^{1-\beta}$ provided 
\begin{align*}
m\geq C\beta^{3/2}\max\lcb\frac{\mu_0}{\vep_0^2}, \frac{\mu_1(1+\vep_0)}{\vep_0}\rcb\kappa n^{3/2}r\log^{3/2}(n).
\end{align*}
\end{theorem}
\subsubsection{Initialization via Resampled Riemannian Gradient Descent and Trimming}
\begin{algorithm}[t]
\caption{Initialization via Resampled Riemannian Gradient Descent and Trimming}\label{alg:init2}
\begin{algorithmic}[]
\STATE\textbf{Partition} $\Omega$  into $L+1$ equal groups: $\Omega_0$, $\cdots$, $\Omega_L$; and the size of each group is denoted by $\tilm$. \\{\bf Set}
$Z_0 = \H_r\lb \frac{\tilm}{n^2}\P_{\Omega_0}(X)\rb$.
\FOR{$l=0,\cdots,L-1$}
\STATE 1. $\widehat{Z}_{l}=\trim(Z_{l})$
\STATE 2. $Z_{l+1}= \H_r\lb \widehat{Z}_{l}+\frac{n^2}{\tilm}\P_{\widehat{T}_{l}}\P_{\Omega_{l+1}}\lb X-\widehat{Z}_{l}\rb\rb$
\ENDFOR
\STATE\textbf{Output}: $X_0=Z_L$
\end{algorithmic}
\end{algorithm}
\begin{algorithm}[t]
\caption{\tt{trim}}\label{alg:trim}
\begin{algorithmic}[]
\STATE{\bf Input:} $Z_l = U_l\Sigma_lV_l^*$
\STATE{\bf Output:} $\widehat{Z}_l=A_l\Sigma_l B_l^*$, where 
\begin{align*}
A_l^{(i)} = \frac{U_l^{(i)}}{\ln U_l^{(i)}\rn}\min\lb \ln U_l^{(i)}\rn,\sqrt{\frac{\mu_0 r}{n}}\rb,\quad B_l^{(i)} = \frac{V_l^{(i)}}{\ln V_l^{(i)}\rn}\min\lb \ln V_l^{(i)}\rn,\sqrt{\frac{\mu_0 r}{n}}\rb
\end{align*}
\end{algorithmic}
\end{algorithm}
The sampling complexity in Thm.~\ref{thm:main_init1} depends on $n^{1.5}$, rather than linearly on $n$.
To attenuate this dependence, we consider a more delicate initialization scheme via the resampled Riemannian gradient descent 
followed by trimming in each iteration, see Alg.~\ref{alg:init2}. The trimming procedure (Alg.~\ref{alg:trim}) projects the estimate onto the set of $\mu_0$-incoherent matrices, while the resampling scheme breaks the dependence between the past iterate and the new sampling set. So we can 
establish the required isometry properties which are needed to prove the linear convergence of the iterates until they reach an order of $p^{1/2}$ neighborhood around the measured matrix where Thms.~\ref{thm:rgrad_det} and \ref{thm:rcg_det} are activated. 
{\it Since the main computational cost in each iteration of  Alg.~\ref{alg:init2} is  $O(|\Omega_l|r)$ flops  \cite{bart2012riemannian}, the total  computational cost of Alg.~\ref{alg:init2} is only slightly larger than one iteration of Alg.~\ref{alg:rgrad} applied on the full sampling set $\Omega$.}
The output of Alg.~\ref{alg:init2} satisfies the following property. 
\begin{lemma}\label{thm:init2}
Let $X\in\R^{n\times n}$ be the measured rank $r$ matrix. Then for all $\beta>1$,  with probability at least $1-2n^{1-\beta}-4Ln^{2-2\beta}$,
\begin{align*}
\ln X_0-X\rn_F\leq \lb\frac{5}{6}\rb^{L}\frac{\sigmin{X}}{256\kappa^2}
\end{align*}
provided $\tilm\geq C\beta\max\lcb\mu_0,\mu_1^2\rcb\kappa^6nr^2\log(n)$.
\end{lemma}
The proof of Lem.~\ref{thm:init2} is presented in Sec.~\ref{subsec:proof_err_init2}. In Alg.~\ref{alg:init2}, we use fixed stepsize $\frac{n^2}{\tilm}$ for ease of exposition, which can be replaced by the adaptive stepsize similar to Alg.~\ref{alg:rgrad}. Lemma~\ref{thm:init2} implies that if we take 
\begin{align*}
L\geq 6\log\lb\frac{\beta n\log(n)}{24\vep_0}\rb,
\end{align*}
 the third condition \eqref{eq:rgrad_det_cond3} and \eqref{eq:rcg_det_cond3} in Thms.~\ref{thm:rgrad_det} and \ref{thm:rcg_det}   can be satisfied with probability at least $1-2n^{1-\beta}-24\log\lb\frac{\beta n\log(n)}{24\vep_0}\rb n^{2-2\beta}$. 
So together with Lem.~\ref{lem:repetition},  Thms.~\ref{thm:isometry}, \ref{thm:rgrad_det} and \ref{thm:rcg_det}, we can establish the following theorem. 
\begin{theorem}[Recovery Guarantee \rom{2}]\label{thm:main_init2}
Let $X\in\R^{n\times n}$ be the measured rank $r$ matrix. Suppose $\Omega$ with $|\Omega|=m$ is a set of indices sampled independently and uniformly with replacement. Let $X_0$ be the output of Alg.~\ref{alg:init2}. Then for all $\beta>1$, the iterates of 
the Riemannian gradient descent algorithm (Alg.~\ref{alg:rgrad}) and the restarted Riemannian conjugate gradient algorithm (Alg.~\ref{alg:rcg}, restarting when either the inequality in \eqref{eq:restart_cond} is violated) with $\kappa_1\kappa_2<1/3$ is guaranteed to converge to $X$ with probability at least  $1-2n^{1-\beta}-\lb24\log\lb\frac{\beta n\log(n)}{24\vep_0}\rb+3\rb n^{2-2\beta}$ provided 
\begin{align*}
m\geq C\beta\max\lcb\frac{\mu_0}{\vep_0^2},\mu_1^2\rcb\kappa^6nr^2\log(n)\log\lb\frac{\beta n\log(n)}{24\vep_0}\rb.
\end{align*}
\end{theorem}
\subsection{Related Work}
From the pioneer work of  Recht, Fazel, and Parrilo \cite{rechtfazelparrilo2010nnm,Fazel_thesis} and Cand\`es and Recht \cite{candesrecht2009mc}, the  low rank matrix reconstruction problem has received intensive investigations  from both the theoretical and algorithmic aspects. Recht, Fazel, and Parrilo \cite{rechtfazelparrilo2010nnm} studied  nuclear norm minimization for the low rank matrix recovery problem where each measurement is obtained by taking inner product between the underlying low rank matrix and  a dense measurement matrix. They showed that if the restricted isometry constant  of the sensing operator is smaller than a numerical constant, then nuclear norm minimization is guaranteed to recover a low rank matrix exactly.
Moreover, this condition can be satisfies for certain family of random measurement matrices provided the number of measurements 
is $O(nr)$ \cite{candesplan2009oracle}.  Many non-convex optimization algorithms have also been studied for low rank matrix recovery based on either the restricted isometry constant of the sensing operator or directly on the random measurement models, see \cite{cevher2012matrixrecipes,leebresler2011admira,tw2012nihtmc,jmd2010svp,procrstes_flow,chenwain_fast,zhenglaff,saolre,bhokyrsan,zhaowangliu,whitesangward} and references therein.  In particular, the Riemannian gradient descent and conjugate gradient descent algorithms  for low rank matrix recovery have been studied in \cite{CGIHT_dense}.

Cand\`es and Recht \cite{candesrecht2009mc} first studied matrix completion by nuclear norm minimization under the Bernoulli  model and incoherent conditions, showing that $O(n^{1.2}r\log(n))$ measurements are sufficient for successful recovery with high probability. This result was subsequently sharpened to $O(nr\log^\alpha(n))$  in \cite{gross2011recoverlowrank,recht2011simple,candestao2009mc,chen_incoherent}. The nuclear norm minimization problem can be solved by either  semidefinite programming \cite{VaBo1996sdp} or iterative soft thresholding algorithms \cite{ccs2010svt}. A gradient descent algorithm on the Grassmannian manifold of low dimensional subspaces was studied in \cite{rhkamso}. It was shown that $O(\kappa^6nr^2)$ number of measurements allows for convergence of the algorithm, which was established by 
penalizing the incoherence  of the estimate in the loss function. 
The proof in \cite{rhkamso} relies on the stationary point analysis and linear convergence rate of the algorithm was not shown.
This framework was further extended in \cite{sunluo} to  gradient descent algorithms based on the product factorization of low rank matrices. Alternating minimization was studied in \cite{kes_thesis, JaNeSa2012ammc} with the best known sampling complexity being $O(\kappa^8nr\log\lb\frac{n}{\epsilon}\rb)$,  where the desired accuracy $\epsilon$ was introduced because the algorithm requires a fresh set of measurements in each iteration. {\it In contrast,  Thm.~\ref{thm:init2}  only requires resampling for the initialization stage and our algorithms can converge linearly to the desired low rank matrix (within arbitrary precision)
for the fixed number of measurements.}
\section{Numerical Experiments}\label{sec:numerics}
%
In this section, we present empirical observations of the Riemannian gradient descent and conjugate 
gradient descent algorithms for random tests. {We test three algorithms, namely, the Riemannian gradient descent algorithm (Alg. \ref{alg:rgrad}), the Riemannian conjugate gradient descent algorithm (Alg. \ref{alg:rcg}), and the Riemannian conjugate gradient algorithm being restarted if one of the conditions in \eqref{eq:restart_cond} is violated. These three algorithms are abbreviated as RGrad, RCG, and RCG restarted respectively.  For RCG restarted,} we take
$\kappa_1=0.1$ and $\kappa_2=1$ {in \eqref{eq:restart_cond}}.  {All tested algorithms are initialized by one step hard thresholding instead of Alg. \ref{alg:init2}, as the 
former has already led to} very good performance (see Figs.~\ref{fig_phase}, \ref{fig_time} and \ref{fig_stability}) and preliminary numerical results
didn't find much difference between {these} two initialization strategies for random simulations. The numerical experiments are conducted on a Mac Pro
laptop with 2.5GHz quad-core Intel Core i7 CPUs and 16 GB memory and executed from Matlab
2014b.  
\subsection{Empirical Phase Transition}\label{subsec:phase}
An important question in matrix completion is how many of measurements are needed in order for an 
algorithm to be able to reliably reconstruct a low rank matrix. We investigate the recovery abilities of 
the tested algorithms in the framework of phase transition, which compares the number of 
measurements, $m$, the size of an $n\times n$ matrix, $n^2$, and the minimum number of measurements 
needed to recover an $n\times n$ rank $r$ matrix\footnote{All   $n\times n$ rank $r$ matrices form a smooth embedded manifold of dimension $(2n-r)r$ in the ambient space \cite{bart2012riemannian}.}, $(2n-r)r$, through the {\it undersampling} and 
{\it oversampling} ratios
\begin{align*}
p = \frac{m}{n^2}, \quad q=\frac{(2n-r)r}{m}.
\end{align*}
The unit square $(p, q)\in[0,1]^2$ defines a phase transition space. Given a triple $(n,m,r)$ corresponding to  
a fixed pair of $(p, q)$, we conducts test on {\em ten} random instances. The test rank $r$ matrix
is formed as the product of two random rank $r$ matrices;  that is $X=LR$, where $L\in\R^{n\times r}$ and 
$R\in\R^{r\times n}$ with the entries of $L$ and $R$ sampled from the standard Gaussian distribution.  The measurement 
vector $\P_\Omega(X)$ is obtained by sampling $m$ entries of $X$ uniformly at random.
 An algorithm is  considered to have 
successfully recovered $X$ if it returns a matrix $X_l$ which satisfies 
\begin{align*}
\frac{\ln X_l-X\rn_F}{\ln X\rn_F}\leq 10^{-2}.
\end{align*}
The tests are conducted with $n=800$ and $p$ taking $18$ equispaced values from $0.1$ to $0.95$.

The probabilities  of successful recovery  for the three tested algorithms RGrad, RCG and RCG restarted are displayed in Fig.~\ref{fig_phase}. In the figure, white color indicates that the algorithm can recover all of the ten random test matrices while black color implies the algorithm fails to recover each of the randomly drawn matrices. For each tested algorithm, a clear phase transition occurs when $ q$ is greater than $0.9$ for all the values of $p$. This implies that  all the three tested algorithms are able to recover a rank $r$ matrix from $m=C\cdot(2n-r)r$ number of measurements with $C$ being slightly larger than 1. In addition, Fig.~\ref{fig_phase_RRCG} also shows the effectiveness of our restarting conditions for the Riemannian conjugate gradient descent algorithm. 
\begin{figure*}[!t]
\centering
\subfloat[]{\includegraphics[width=2in]{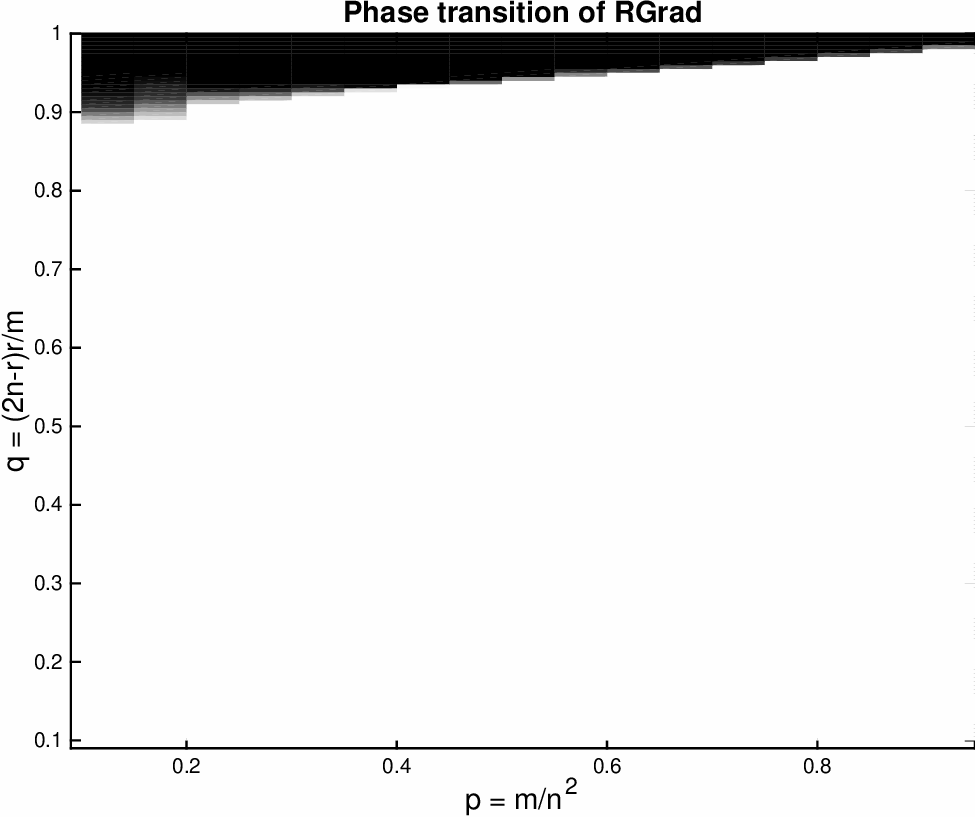}
\label{fig_phase_RGD}}
\hfil
\subfloat[]{\includegraphics[width=2in]{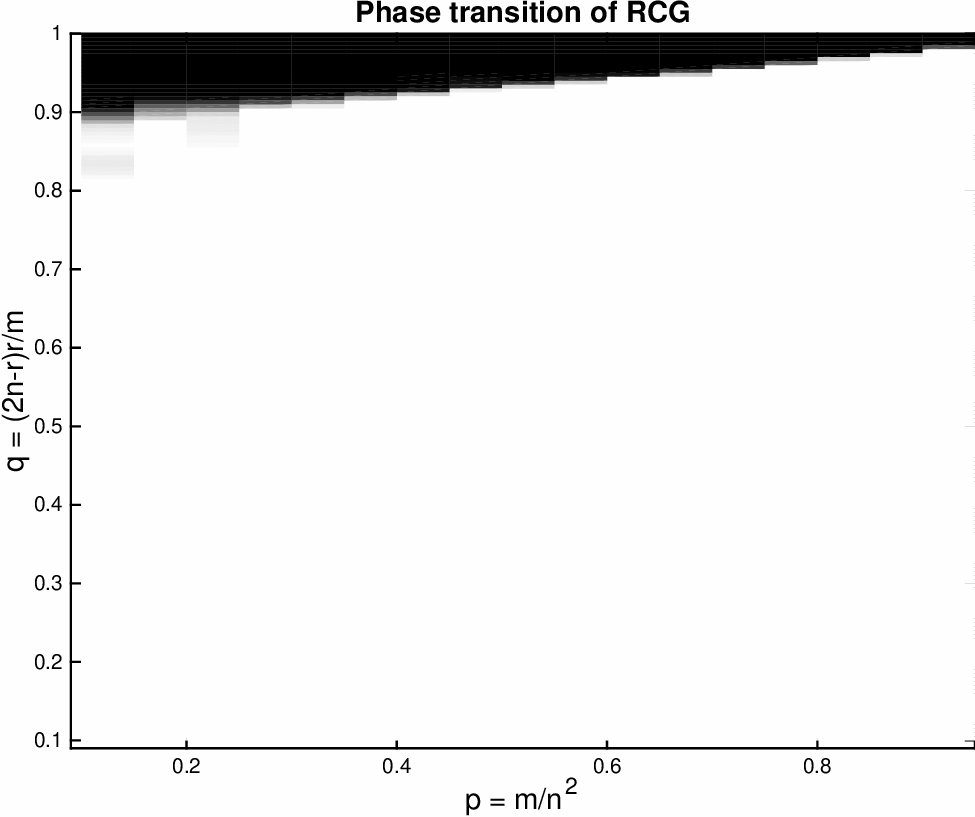}
\label{fig_phase_RCG}}
\hfil
\subfloat[]{\includegraphics[width=2in]{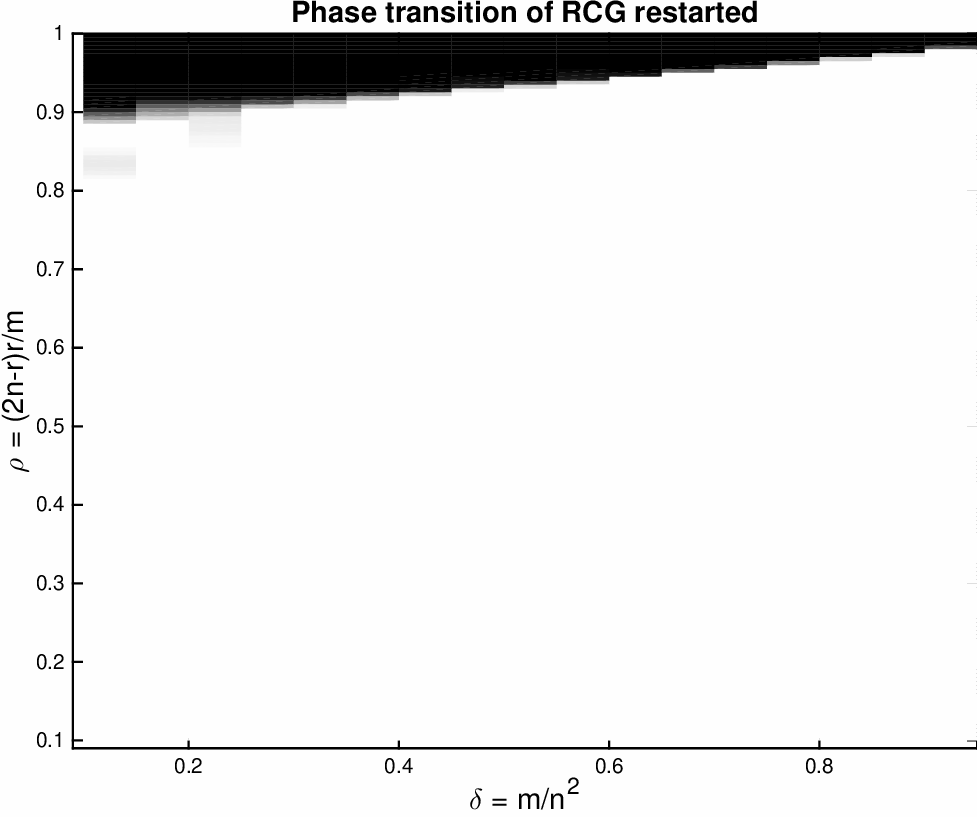}
\label{fig_phase_RRCG}}
\caption{Empirical phase transition curves (a) RGrad, (b) RCG and (c) RCG restarted when $n=800$. Horizontal axis $p=m/n^2$ and 
vertical axis $ q =(2n-r)r/m$. White denotes successful recovery in all ten random tests, and black denotes failure in all tests.}
\label{fig_phase}
\end{figure*}
\subsection{Computation Efficiency}
We compare the computational efficiency of RGrad, RCG, RCG restarted by conducting random tests on matrices of $8000\times 8000$ and  
 rank $100$. The tests are conducted with two different oversampling ratios $1/ q\in\lcb 2,3\rcb$. The algorithms are terminated when the relative residual is less than $10^{-9}$. The relative residual plotted against the number of iterations and the average recovery time are presented in Fig.~\ref{fig_time}. First it can be observed that the convergence
curves for RCG and RCG restarted are almost indistinguishable, differing only in one
or two iterations. This again shows the effectiveness of our restarting conditions.
A close look at the computational results reveals that restarting usually occurs in
the first few iterations for RCG restarted. Moreover, RCG and RCG restarted are
sufficiently faster than RGrad both in terms of  the number of iterations and  in terms of 
the average computation time. It takes less number of iterations for each algorithm to converge below the desired accuracy 
when the number of measurements increases.
 \begin{figure*}[!t]
\centering
\subfloat[]{\includegraphics[width=2.5in]{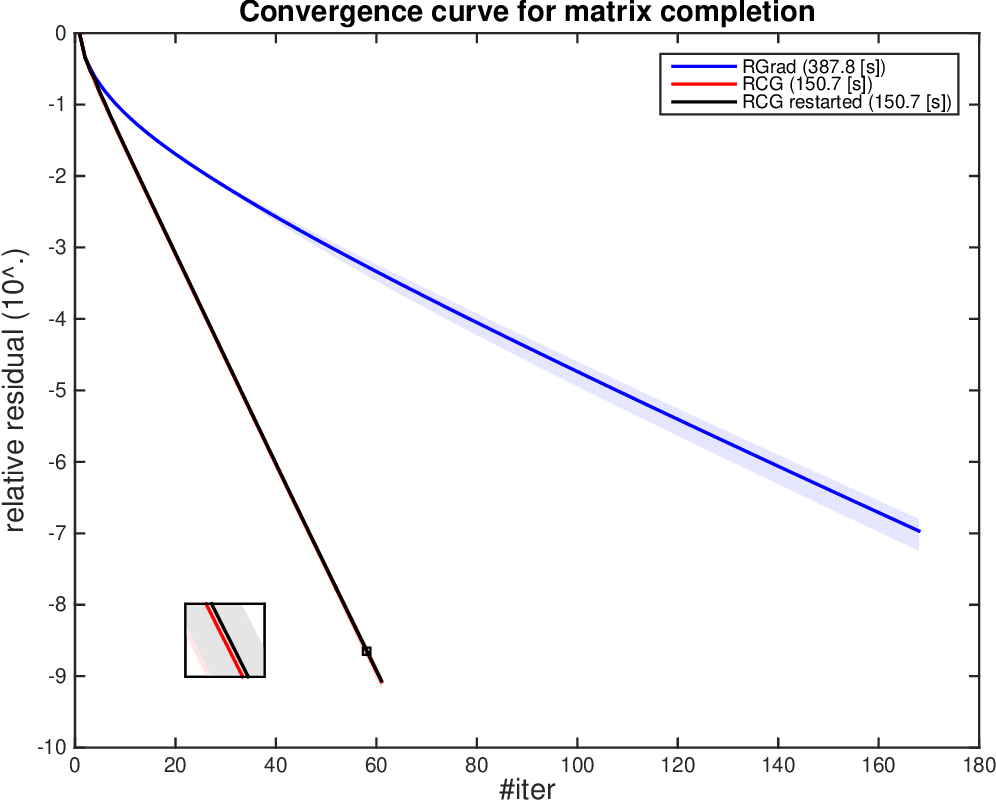}
\label{fig_time_2}}
\hfil
\subfloat[]{\includegraphics[width=2.5in]{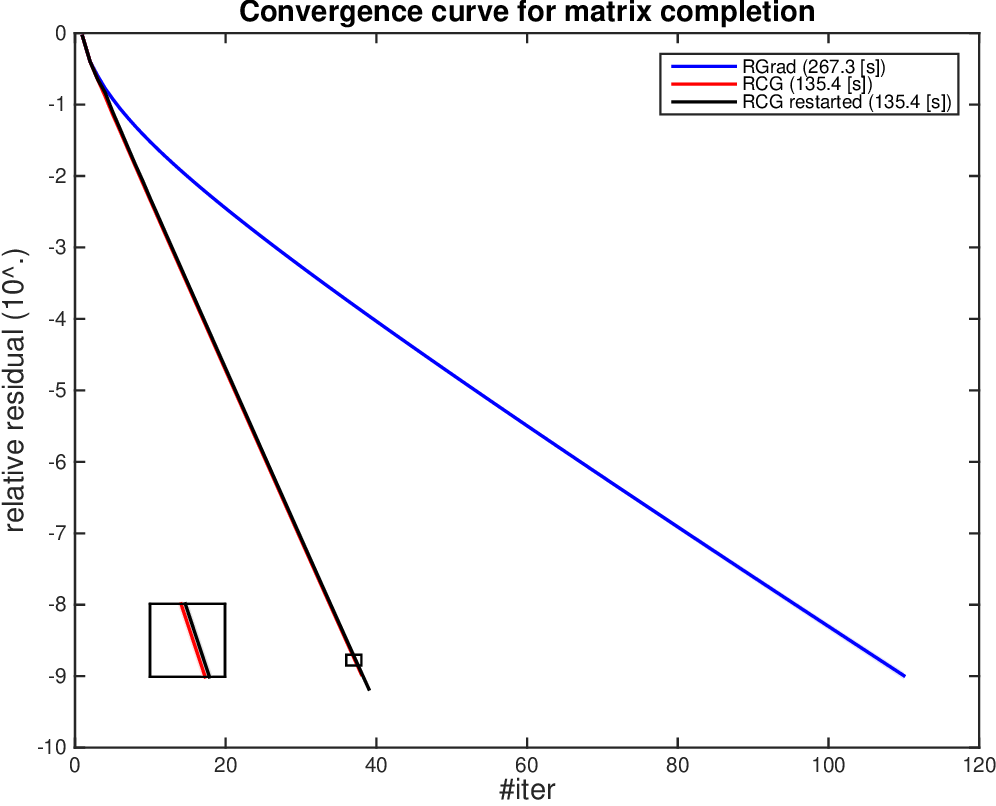}
\label{fig_time_3}}
\caption{Relative residual (mean and standard deviation over ten random tests) as function of
number of iterations for $n=8000$, $r=100$, $1/ q=2$ (a) and $1/ q=3$ (b). The values after each algorithm are the average computational time
(seconds) for convergence.}
\label{fig_time}
\end{figure*}
\subsection{Robustness to Additive Noise}\label{subsec:noise}
Following the test set-up in last subsection, we further explore performance of the algorithms 
under the measurements with additive noise. Tests with additive noise have  the sampled 
entries $\P_\Omega(X)$ corrupted by the vector
\begin{align*}
e=\sigma\cdot\ln\P_\Omega(X)\rn_F\cdot\frac{w}{\ln w\rn_2},
\end{align*}
where the entries of $w$ are i.i.d standard Gaussian random variables and $\sigma$
is referred to as noise level. We conduct tests with {\it nine} different values of $\sigma$
from $10^{-4}$ to $1$; and for each $\sigma$, {\it ten} random tests are conducted.
The average relative reconstruction error in dB plotted against the signal-to-noise ratio (SNR) is presented 
in Fig.~\ref{fig_stability} for RGrad, RCG and RCG restarted respectively. The plots clearly show the 
desirable linear scaling between the noise levels and the relative errors for all the three tested algorithms.
The relative error decreases as the number of measurements increases.
\begin{figure*}[!t]
\centering
\subfloat[]{\includegraphics[width=2in]{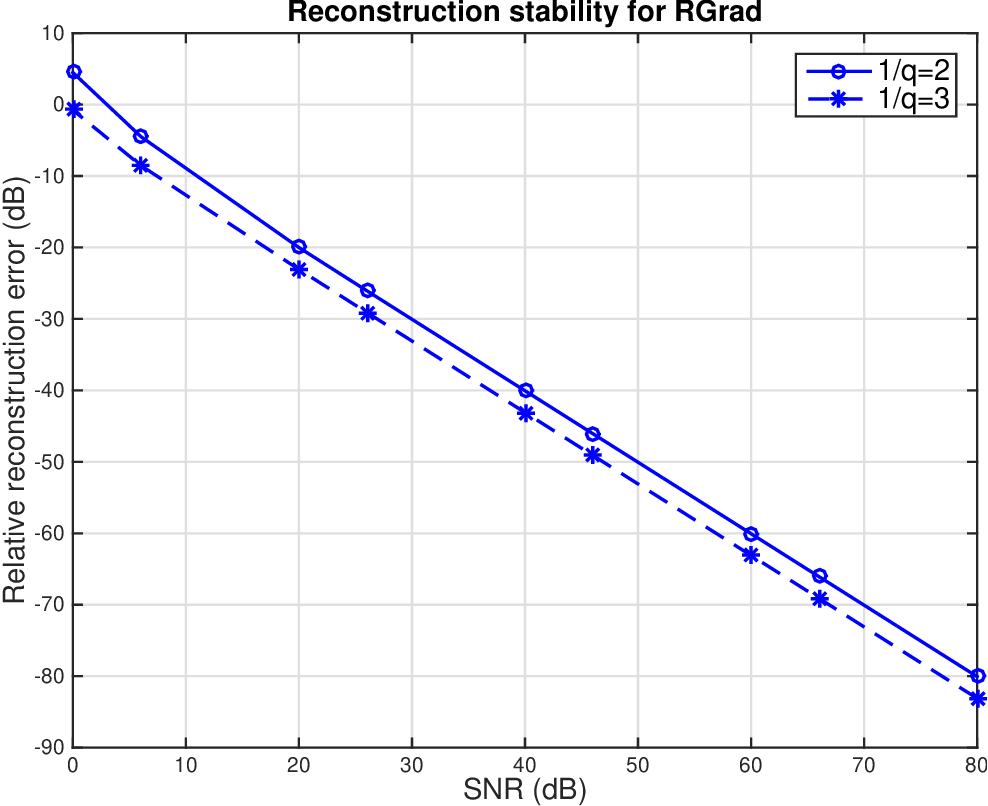}
\label{fig_stability_RGD}}
\hfil
\subfloat[]{\includegraphics[width=2in]{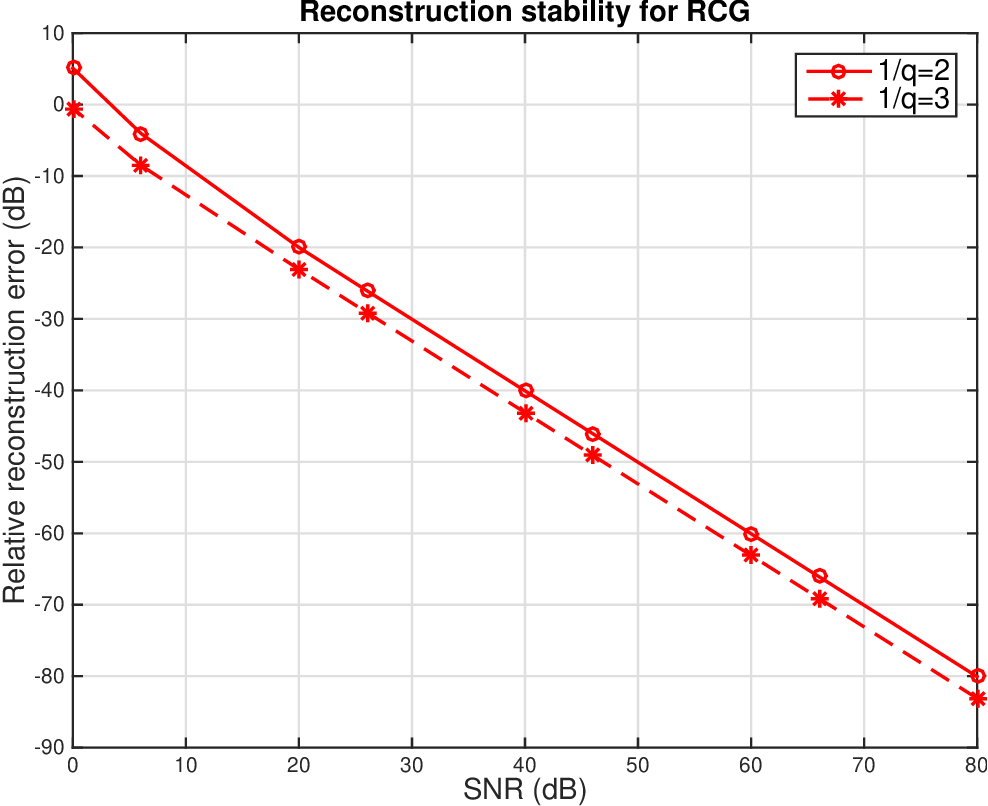}
\label{fig_stability_RCG}}
\hfil
\subfloat[]{\includegraphics[width=2in]{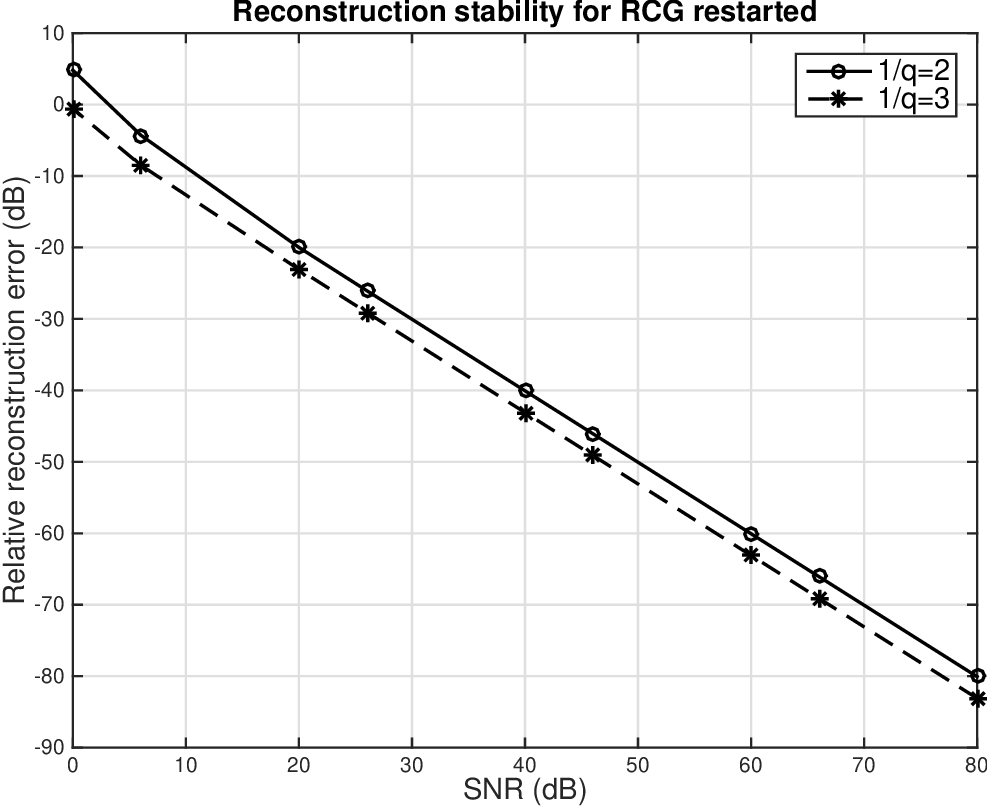}
\label{fig_stability_RRCG}}
\caption{Performance of (a) RGrad, (b) RCG and (c) RCG restarted under different SNR.}
\label{fig_stability}
\end{figure*}

\section{Proofs}\label{sec:proofs}
In this section, we prove the main results in Sec.~2. We first list several  technical lemmas which are of independent interest. The proofs of these lemmas are presented in Appendix~\ref{sec:proofs_of_lemmas}.
\subsection{Technical Lemmas}
\begin{lemma}[Bounds for  Projections]\label{lem:tech1}
Let $X_l=U_l\Sigma_l V_l^*$ be a rank $r$ matrix, and $T_l$ be the tangent space of the rank $r$ matrix manifold 
at $X_l$.  Let $X=U\Sigma V^*$ be another $r$ matrix, and $T$ be the corresponding tangent space. Then 
\begin{align}\label{eq:tech1}
\begin{split}
&\ln U_lU_l^*-UU^*\rn\leq\frac{\ln X_l-X\rn_F}{\sigmin{X}},\quad\ln V_lV_l^*-VV^*\rn\leq\frac{\ln X_l-X\rn_F}{\sigmin{X}},\\
&\ln U_lU_l^*-UU^*\rn_F\leq\frac{\sqrt{2}\ln X_l-X\rn_F}{\sigmin{X}},\quad\ln V_lV_l^*-VV^*\rn_F\leq\frac{\sqrt{2}\ln X_l-X\rn_F}{\sigmin{X}},\\
&\ln\lb \I-\P_{T_l}\rb X\rn_F\leq\frac{\ln X_l-X\rn_F^2}{\sigmin{X}},\quad\ln\P_{T_l}-\P_T\rn\leq\frac{2\ln X_l-X\rn_F}{\sigmin{X}}.
\end{split}
\end{align}
\end{lemma}
\begin{lemma}[Restricted Isometry Property in a Local Neighborhood]\label{lem:tech2}
Assume 
\begin{align*}
&\ln\P_\Omega\rn\leq\frac{8}{3}\beta\log(n),\quad\ln \P_T-p^{-1}\P_T\P_\Omega\P_T\rn\leq \varepsilon_0,\quad\mbox{and}\quad\frac{\donr{X_l-X}_F}{\sigmin{X}}\leq\frac{3p^{1/2}\varepsilon_0}{16\beta\log(n)(1+\varepsilon_0)}
\end{align*}
for some $0<\varepsilon_0<1$ and $\beta>1$.
Then 
\begin{align*}
&\ln \P_\Omega\P_{T_l}\rn\leq{\frac{8}{3}\beta\log(n)(1+\varepsilon_0)}p^{1/2},\numberthis\label{eq:tech2_eq1}\\
&\donr{\P_{T_l}-p^{-1}\P_{T_l}\P_\Omega\P_{T_l}}\leq 4\varepsilon_0\numberthis\label{eq:tech2_eq2}.
\end{align*}
\end{lemma}
\begin{lemma}[Asymmetric Restricted Isometry Property]\label{lem:tech3}
Let $X_l=U_l\Sigma_lV_l^*$ and $X=U\Sigma V^*$ be two fixed rank $r$ matrices. Assume
\begin{align*}
&\ln\P_{U_l}\lb e_i\rb\rn\leq\sqrt{\frac{\mu r}{n}}\quad\mbox{and}\quad\ln\P_{V_l}\lb e_j\rb\rn\leq\sqrt{\frac{\mu r}{n}},\\
&\ln\P_U\lb e_i\rb\rn\leq\sqrt{\frac{\mu r}{n}}\quad\mbox{and}\quad\ln\P_V\lb e_j\rb\rn\leq\sqrt{\frac{\mu r}{n}}
\end{align*}
for $1\leq i, j\leq n$. Suppose $\Omega$ with $|\Omega|=m$ is a set of indices sampled independently and uniformly with replacement. Then for any $\beta>1$, 
\begin{align*}
&\ln \frac{n^2}{m}\P_{T_l}\P_{\Omega}\lb\P_U-\P_{U_l}\rb-\P_{T_l}\lb\P_U-\P_{U_l}\rb\rn\leq\sqrt{\frac{48\beta \mu nr \log(n)}{m}}
\end{align*}
with probability at least $1-2n^{2-2\beta}$ provided $m\geq 15\beta\mu nr\log(n)$.
\end{lemma}
\begin{lemma}[A Recursive Relationship]\label{lem:tech4}
Let $\rho_1,~\rho_2$ and $\gamma$ be positive constants satisfying $\rho_2\geq \rho_1$. Define
\begin{align*}
\tau_1=\rho_1+\gamma,~\tau_2=(\rho_2-\rho_1)\gamma, ~\mbox{and}~\nu=\frac{1}{2}\lb\tau_1+\sqrt{\tau_1^2+4\tau_2}\rb.
\end{align*}
Let $\lcb c_l\rcb_{l\geq 0}$ be a non-negative sequence satisfying $c_1\leq \nu c_0$ and 
\begin{align*}
c_{l+1}\leq\rho_1 c_l+\rho_2\sum_{j=0}^{l-1}\gamma^{l-j}c_j,\quad\forall~ l\geq 1. 
\end{align*}
Then if $\tau_1+\tau_2<1$, we have $\nu<1$ and 
\begin{align*}
c_{l+1}\leq \nu^{l+1} c_0.\numberthis\label{eq:tech4_eq1}
\end{align*}
\end{lemma}
\begin{lemma}[Chordal and Projection Distances]\label{lem:tech5}
Let $U_l,~U\in\R^{n\times r}$ be two orthogonal matrices. Then there exists a $r\times r$ unitary matrix $Q$ such that
\begin{align*}
\ln U_l-UQ\rn_F\leq\ln U_lU_l^*-UU^*\rn_F.
\end{align*}
\end{lemma}
\subsection{Proof of Theorem~\ref{thm:rgrad_det}}\label{subsec:proof_rgrad_det}
We start with a  lemma which bounds the search stepsize $\alpha_l$ in Alg.~\ref{alg:rgrad}. 
\begin{lemma}\label{lem:rgrad_alpha}
Assume $\ln\P_{T_l}-p^{-1}\P_{T_l}\P_\Omega\P_{T_l}\rn\leq 4\vep_0$. Then the stepsize $\alpha_l$ in Alg.~\ref{alg:rgrad} can be bounded as 
\begin{align*}
\frac{1}{(1+4\vep_0)p}\leq\alpha_l=\frac{\ln\P_{T_l}\lb G_l\rb\rn_F^2}{\la \P_{T_l}\lb G_l\rb, \P_\Omega\P_{T_l}\lb G_l\rb\ra}\leq\frac{1}{(1-4\vep_0)p}.
\end{align*}
\end{lemma}
\begin{proof}
First the assumption $\ln\P_{T_l}-p^{-1}\P_{T_l}\P_\Omega\P_{T_l}\rn\leq 4\vep_0$ implies
$$
\ln \P_{T_l}\P_\Omega\P_{T_l}\rn\leq p\ln\P_{T_l}-p^{-1}\P_{T_l}\P_\Omega\P_{T_l}\rn + p\ln\P_{T_l}\rn
\leq(1+4\vep_0)p.
$$
On one hand, 
\begin{align*}
&\la \P_{T_l}\lb G_l\rb, \P_\Omega\P_{T_l}\lb G_l\rb\ra= \la \P_{T_l}\lb G_l\rb, \P_{T_l}\P_\Omega\P_{T_l}\lb G_l\rb\ra\leq (1+4\vep_0)p\ln \P_{T_l}\lb G_l\rb\rn_F^2,
\end{align*}
which implies the lower bound of $\alpha_l$. 

On the other hand,
\begin{align*}
\ln\P_{T_l}(G_l)\rn_F^2&=\la\P_{T_l}(G_l) , \P_{T_l}(G_l)-p^{-1}\P_{T_l}\P_\Omega\P_{T_l}(G_l)\ra+\la \P_{T_l}(G_l), p^{-1}\P_{T_l}\P_\Omega\P_{T_l}(G_l)\ra\\
&{\leq 4\vep_0\ln \P_{T_l}(G_l)\rn_F^2+p^{-1}\la \P_{T_l}(G_l), \P_\Omega\P_{T_l}(G_l)\ra},\numberthis\label{eq:PoPTG_lowerbound}
\end{align*}
{from which we can obtain the upper bound of $\alpha_l$ by multiplying $p$ on both sides of the above inequality followed by the rearrangement.}
\end{proof}

The next lemma bounds the local isometry of $\P_\Omega$ when the adaptive stepsize $\alpha_l$ is used. 
\begin{lemma}\label{lem:rgrad_iso2}
Assume $\ln\P_{T_l}-p^{-1}\P_{T_l}\P_\Omega\P_{T_l}\rn\leq 4\vep_0$ and $\alpha_l$ can be bounded as in Lem.~\ref{lem:rgrad_alpha}. Then the spectral norm of $\P_{T_l}-\alpha_l\P_{T_l}\P_\Omega\P_{T_l}$ can be bounded as 
\begin{align*}
\ln \P_{T_l}-\alpha_l\P_{T_l}\P_\Omega\P_{T_l}\rn\leq\frac{8\vep_0}{1-4\vep_0}.
\end{align*}
\end{lemma}
\begin{proof}
The lemma follows from direct calculations 
\begin{align*}
\ln \P_{T_l}-\alpha_l\P_{T_l}\P_\Omega\P_{T_l}\rn&\leq \ln \P_{T_l}-p^{-1}\P_{T_l}\P_\Omega\P_{T_l}\rn+|\alpha_l-p^{-1}|\ln \P_{T_l}\P_\Omega\P_{T_l}\rn\\
&{\leq 4\vep_0+\frac{4\vep_0}{(1-4\vep_0)p}(1+4\vep_0)p}\\
&=\frac{8\vep_0}{1-4\vep_0},
\end{align*}
where the second inequality follows from the assumption and Lem.~\ref{lem:rgrad_alpha}.
\end{proof}

\begin{proof}[Proof of Theorem~\ref{thm:rgrad_det}] 
We first assume that in the $l$-th iteration $X_l$ satisfies
\begin{align*}
\frac{\ln X_l-X\rn_F}{\sigmin{X}}\leq\frac{3p^{1/2}\vep_0}{16\beta\log(n)(1+\vep_0)}\numberthis\label{eq:rgrad_extra_assump}
\end{align*}
 So together with the first two assumptions {\eqref{eq:rgrad_det_cond1} and \eqref{eq:rgrad_det_cond2}} of Thm.~\ref{thm:rgrad_det}, the application of Lem.~\ref{lem:tech2} gives
\begin{align*}
\ln\P_\Omega\P_{T_l}\rn\leq\frac{8}{3}\beta\log(n)(1+\vep_0)p^{1/2}\numberthis\label{eq:rgrad_popt_bd}
\end{align*}
and
\begin{align*}
\ln\P_{T_l}-p^{-1}\P_{T_l}\P_\Omega\P_{T_l}\rn\leq 4\vep_0,\numberthis\label{eq:rgrad_iso1}
\end{align*}
which means the assumptions of Lems.~\ref{lem:rgrad_alpha} and \ref{lem:rgrad_iso2} are satisfied.

Recall that $W_l= X_l+\alpha_l\P_{T_l}(G_l)$ in Alg.~\ref{alg:rgrad}. The proof 
begins with the following inequality 
\begin{align*}
\ln X_{l+1}-X\rn_F&\leq \ln X_{l+1}-W_l\rn_F+\ln W_l-X\rn_F\leq 2\ln W_l-X\rn_F,
\end{align*}
where the second inequality follows from the fact that $X_{l+1}$ is the best
rank $r$ approximation of $W_l$. 
Plugging in  $W_l= X_l+\alpha_l\P_{T_l}(G_l)$ gives
\begin{align*}
\ln X_{l+1}-X\rn_F&\leq 2\ln X_l+\alpha_l\P_{T_l}(G_l)-X\rn_F\\
&=2\ln X_l-X-\alpha_l\P_{T_l}\P_\Omega(X_l-X)\rn_F\\
&\leq2\ln (\P_{T_l}-\alpha_l\P_{T_l}\P_\Omega\P_{T_l})(X_l-X) \rn_F\\
&\quad+2\ln (\I-\P_{T_l})(X_l-X)\rn_F\\
&\quad+2|\alpha_l|\ln \P_{T_l}\P_\Omega(\I-\P_{T_l})(X_l-X)\rn_F\\
&:=I_1+I_2+I_3.\numberthis\label{eq:rgrad_recursive}
\end{align*}
First we have 
\begin{align*}
I_1\leq\frac{16\vep_0}{1-4\vep_0}\ln X_l-X\rn_F
\end{align*}
and 
\begin{align*}
I_2&\leq\frac{2\ln X_l-X\rn_F^2}{\sigmin{X}}\leq\frac{3p^{1/2}\vep_0}{8\beta\log(n)(1+\vep_0)}\ln X_l-X\rn_F\leq\frac{\vep_0}{1-4\vep_0}\ln X_l-X\rn_F
\end{align*}
which follows from Lems.~\ref{lem:rgrad_iso2} and \ref{lem:tech1} respectively.
The third term $I_3$ can be bounded as follows.
\begin{align*}
I_3&\leq 2|\alpha_l|\ln\P_\Omega\P_{T_l}\rn\ln (\I-\P_{T_l})(X)\rn_F\\
&\leq\frac{2}{(1-4\vep_0)p}\frac{8}{3}\beta\log(n)(1+\vep_0)p^{1/2}\frac{\ln X_l-X\rn_F^2}{\sigmin{X}}\\
&\leq\frac{\vep_0}{1-4\vep_0}\ln X_l-X\rn_F,
\end{align*}
where the second inequality follows from \eqref{eq:rgrad_popt_bd}, and Lems.~\ref{lem:rgrad_alpha} and \ref{lem:tech1}, and 
the third inequality follows from \eqref{eq:rgrad_extra_assump}.

So inserting the bounds for $I_1$, $I_2$ and $I_3$ into \eqref{eq:rgrad_recursive} gives 
\begin{align*}
\ln X_{l+1}-X\rn_F\leq \nu_g\ln X_l-X\rn_F,\numberthis\label{eq:rgrad_contract}
\end{align*}
where 
\begin{align*}
\nu_g=\frac{18\vep_0}{1-4\vep_0}<1
\end{align*}
by the assumption of Thm.~\ref{thm:rgrad_det}. 

{It only remains to verify \eqref{eq:rgrad_extra_assump}. By the  assumption of Thm.~\ref{thm:rgrad_det}, \eqref{eq:rgrad_extra_assump} is valid for $l=0$. Since $\ln X_l-X\rn_F$ is a contractive sequence following from \eqref{eq:rgrad_contract}, \eqref{eq:rgrad_extra_assump} is valid for all $l\geq 0$ by induction.}
\end{proof}
\subsection{Proof of Theorem~\ref{thm:rcg_det}}\label{subsec:proof_rcg_det}
We first estimate $\alpha_l$ and $\beta_l$ in Alg.~\ref{alg:rcg} with the restarting conditions in \eqref{eq:restart_cond}.
\begin{lemma}\label{lem:rcg_alpha_beta}
Assume $\ln\P_{T_l}-p^{-1}\P_{T_l}\P_\Omega\P_{T_l}\rn\leq 4\vep_0$.  When restarting occurs, $\beta_l=0$ and 
$\alpha_l$ can be bounded as in Lem.~\ref{lem:rgrad_alpha}; otherwise
\begin{align*}
|\beta_l|\leq\vep_\beta\quad\mbox{and}\quad |\alpha_l\cdot p-1|\leq\vep_\alpha,
\end{align*}
where 
\begin{align*}
\vep_\beta=\frac{4\kappa_2\vep_0}{1-4\vep_0}+\frac{\kappa_1\kappa_2}{1-4\vep_0},~\vep_\alpha=\frac{4\vep_0}{(1-4\vep_0)-\kappa_1\lb1+4\vep_0\rb}.
\end{align*}
\end{lemma}
Note that the bounds above are also valid even when restarting occurs since $\kappa_1\ge 0$ and $\kappa_2\ge 0$.
\begin{proof}
The orthogonalization weight $\beta_l$ can be bounded as follows
\begin{align*}
|\beta_l|&=\lab \frac{\la\P_{T_l}\lb G_l\rb,\P_\Omega\P_{T_l}\lb P_{l-1}\rb\ra}{\la\P_{T_l}\lb P_{l-1}\rb, \P_\Omega\P_{T_l}\lb P_{l-1}\rb\ra}\rab\\
&\leq\lab \frac{\la\P_{T_l}\lb G_l\rb,(\P_{T_l}\P_\Omega\P_{T_l}-p\P_{T_l})\lb P_{l-1}\rb\ra}{\la\P_{T_l}\lb P_{l-1}\rb, \P_\Omega\P_{T_l}\lb P_{l-1}\rb\ra}\rab+\lab \frac{p\la\P_{T_l}\lb G_l\rb,\P_{T_l}\lb P_{l-1}\rb\ra}{\la\P_{T_l}\lb P_{l-1}\rb, \P_\Omega\P_{T_l}\lb P_{l-1}\rb\ra}\rab\\
&\leq\frac{4\vep_0p}{(1-4\vep_0)p}\frac{\ln\P_{T_l}(G_l)\rn_F}{\ln\P_{T_l}(P_{l-1})\rn_F}+\frac{p}{(1-4\vep_0)p}\frac{\lab \la\P_{T_l}\lb G_l\rb,\P_{T_l}\lb P_{l-1}\rb\ra\rab}{\ln\P_{T_l}(P_{l-1})\rn_F^2}\\
&\leq\frac{4\kappa_2\vep_0}{1-4\vep_0}+\frac{1}{1-4\vep_0}\frac{\lab \la\P_{T_l}\lb G_l\rb,\P_{T_l}\lb P_{l-1}\rb\ra\rab}{\ln\P_{T_l}(G_l)\rn_F\ln\P_{T_l}(P_{l-1})\rn_F}\frac{\ln\P_{T_l}(G_l)\rn_F}{\ln\P_{T_l}(P_{l-1})\rn_F}\\
&\leq\frac{4\kappa_2\vep_0}{1-4\vep_0}+\frac{\kappa_1\kappa_2}{1-4\vep_0},
\end{align*}
where in the third line we used $\la\P_{T_l}\lb P_{l-1}\rb, \P_\Omega\P_{T_l}\lb P_{l-1}\rb\ra\geq (1-4\vep_0)p\ln\P_{T_l}(P_{l-1})\rn_F^2$ deducted from \eqref{eq:PoPTG_lowerbound}, and the last two inequalities follow from the restarting conditions.

To bound $\alpha_l$, we first need to bound $\ln \P_{T_l}(G_l)\rn_F$ in terms of $\ln\P_{T_l}(P_l)\rn_F$. Note that 
\begin{align*}
|\beta_l\la \P_{T_l}(G_l), \P_{T_l}(P_{l-1})\ra|&=\lab \frac{\la\P_{T_l}\lb G_l\rb,\P_{T_l}\P_\Omega\P_{T_l}\lb P_{l-1}\rb\ra}{\la\P_{T_l}\lb P_{l-1}\rb,\P_{T_l} \P_\Omega\P_{T_l}\lb P_{l-1}\rb\ra}\la \P_{T_l}(G_l), \P_{T_l}(P_{l-1})\ra\rab\\
&\leq\frac{(1+4\vep_0)p\ln\P_{T_l}(G_l)\rn_F}{(1-4\vep_0)p\ln\P_{T_l}(P_{l-1})\rn_F}\lab \la \P_{T_l}(G_l), \P_{T_l}(P_{l-1})\ra\rab\\
&\leq\frac{1+4\vep_0}{1-4\vep_0}\ln\P_{T_l}(G_l)\rn_F^2\frac{\lab \la \P_{T_l}(G_l), \P_{T_l}(P_{l-1})\ra\rab}{\ln\P_{T_l}(G_l)\rn_F\ln\P_{T_l}(P_{l-1})\rn_F}\\
&\leq\frac{\kappa_1(1+4\vep_0)}{1-4\vep_0}\ln \P_{T_l}(G_l)\rn_F^2.
\end{align*}
So 
\begin{align*}
\lab\la\P_{T_l}(P_l),\P_{T_l}(G_l)\ra\rab&=\lab\la \P_{T_l}(G_l)+\beta_l\P_{T_l}(P_{l-1}),\P_{T_l}(G_l)\ra\rab\\
&\geq\ln \P_{T_l}(G_l)\rn_F^2-|\beta_l\la \P_{T_l}(G_l), \P_{T_l}(P_{l-1})\ra|\\
&\geq\lb 1-\frac{\kappa_1(1+4\vep_0)}{1-4\vep_0}\rb\ln \P_{T_l}(G_l)\rn_F^2.
\end{align*}
The application of the Cauchy-Schwarz inequality gives 
\begin{align*}
\ln\P_{T_l}(G_l)\rn_F\leq \frac{1}{1-\frac{\kappa_1(1+4\vep_0)}{1-4\vep_0}}\ln\P_{T_l}(P_l)\rn_F.
\end{align*}
Since $\alpha_l$ can be rewritten as 
\begin{align*}
\alpha_l&=\frac{\la \P_{T_l}\lb G_l\rb, \P_{T_l}\lb P_l\rb\ra}{\la \P_{T_l}\lb P_l\rb, \P_\Omega\P_{T_l}\lb P_l\rb\ra}\\
&=\frac{\la \P_{T_l}\lb G_l\rb, p^{-1}\P_{T_l}\P_\Omega\P_{T_l}\lb P_l\rb\ra}{\la \P_{T_l}\lb P_l\rb, \P_\Omega\P_{T_l}\lb P_l\rb\ra}+\frac{\la \P_{T_l}\lb G_l\rb, (\P_{T_l}-p^{-1}\P_{T_l}\P_\Omega\P_{T_l})\lb P_l\rb\ra}{\la \P_{T_l}\lb P_l\rb, \P_\Omega\P_{T_l}\lb P_l\rb\ra}\\
&=\frac{\la \P_{T_l}\lb G_l\rb, p^{-1}\P_\Omega\P_{T_l}\lb P_l\rb\ra}{\la \P_{T_l}\lb P_l\rb, \P_\Omega\P_{T_l}\lb P_l\rb\ra}+\frac{\la \P_{T_l}\lb G_l\rb, (\P_{T_l}-p^{-1}\P_{T_l}\P_\Omega\P_{T_l})\lb P_l\rb\ra}{\la \P_{T_l}\lb P_l\rb, \P_\Omega\P_{T_l}\lb P_l\rb\ra}\\
&=p^{-1}+\frac{\la \P_{T_l}\lb G_l\rb, (\P_{T_l}-p^{-1}\P_{T_l}\P_\Omega\P_{T_l})\lb P_l\rb\ra}{\la \P_{T_l}\lb P_l\rb, \P_\Omega\P_{T_l}\lb P_l\rb\ra},
\end{align*}
it can be bounded as follows
\begin{align*}
|\alpha_l\cdot p-1|&\leq p\lab \frac{\la \P_{T_l}\lb G_l\rb, (\P_{T_l}-p^{-1}\P_{T_l}\P_\Omega\P_{T_l})\lb P_l\rb\ra}{\la \P_{T_l}\lb P_l\rb, \P_{T_l}\P_\Omega\P_{T_l}\lb P_l\rb\ra}\rab\\
&\leq \frac{4\vep_0p}{(1-4\vep_0)p}\frac{\ln \P_{T_l}\lb G_l\rb\rn_F}{\ln\P_{T_l}(P_l)\rn_F}\\
&\leq\frac{4\vep_0}{(1-4\vep_0)-\kappa_1(1+4\vep_0)},
\end{align*}
which completes the proof.
\end{proof}
The following lemma bounds the local isometry of $\P_\Omega$ when the adaptive 
stepsize $\alpha_l$ is used.
\begin{lemma}\label{lem:rcg_iso2}
Assume $\ln\P_{T_l}-p^{-1}\P_{T_l}\P_\Omega\P_{T_l}\rn\leq 4\vep_0$, and $\alpha_l$ is bounded as in Lem.~\ref{lem:rcg_alpha_beta}. Then the spectral norm of $\P_{T_l}-\alpha_l\P_{T_l}\P_\Omega\P_{T_l}$ can be bounded 
as 
\begin{align*}
\ln \P_{T_l}-\alpha_l\P_{T_l}\P_\Omega\P_{T_l}\rn\leq 4\vep_0+\vep_\alpha(1+4\vep_0).
\end{align*}
\end{lemma}
\begin{proof}
Direct calculations give 
\begin{align*}
\ln \P_{T_l}-\alpha_l\P_{T_l}\P_\Omega\P_{T_l}\rn&\leq \ln \P_{T_l}-p^{-1}\P_{T_l}\P_\Omega\P_{T_l}\rn+|\alpha_l-p^{-1}|\ln \P_{T_l}\P_\Omega\P_{T_l}\rn\\
&\leq 4\vep_0+|\alpha\cdot p-1|\ln p^{-1}\P_{T_l}\P_\Omega\P_{T_l}\rn\\
&\leq 4\vep_0+\vep_\alpha(1+4\vep_0),
\end{align*}
which completes the proof.
\end{proof}
\begin{proof}[Proof of Theorem~\ref{thm:rcg_det}]
We first assume that for all $j\leq l$
\begin{align*}
\frac{\ln X_j-X\rn_F}{\sigmin{X}}\leq\frac{3p^{1/2}\vep_0}{16\beta\log(n)(1+\vep_0)}.\numberthis\label{eq:rcg_extra_assump}
\end{align*}
So together with the first two assumptions \eqref{eq:rgrad_det_cond1} and \eqref{eq:rgrad_det_cond2} of Thm.~\ref{thm:rcg_det}, the application of Lem.~\ref{lem:tech2} gives
\begin{align*}
\ln\P_\Omega\P_{T_j}\rn\leq\frac{8}{3}\beta\log(n)(1+\vep_0)p^{1/2}\numberthis\label{eq:rcg_popt_bd}
\end{align*}
and
\begin{align*}
\ln\P_{T_j}-p^{-1}\P_{T_j}\P_\Omega\P_{T_j}\rn\leq 4\vep_0,\numberthis\label{eq:rcg_iso1}
\end{align*}
which means the assumptions of Lems.~\ref{lem:rcg_alpha_beta} and \ref{lem:rcg_iso2} are satisfied for all $j\leq l$.

Analogous to \eqref{eq:rgrad_recursive}, we have 
\begin{align*}
\ln X_{l+1}-X\rn_F&\leq2\ln X_l+\alpha_lP_l-X \rn_F\\
&=2\ln X_l-X-\alpha_l\P_{T_l}\P_\Omega(X_l-X)+\alpha_l\beta_l\P_{T_l}(P_{l-1})\rn_F\\
&\leq 2\ln (\P_{T_l}-\alpha_l P_{T_l}\P_\Omega\P_{T_l})(X_l-X)\rn_F\\
&\quad+2\ln (\I-\P_{T_l})(X_l-X)\rn_F\\
&\quad+2|\alpha_l|\ln P_{T_l}\P_\Omega(\I-\P_{T_l})(X_l-X)\rn_F\\
&\quad+2|\alpha_l||\beta_l|\ln\P_{T_l}(P_{l-1})\rn_F\\
&:=I_4+I_5+I_6+I_7.\numberthis\label{eq:rcg_recursive}
\end{align*}

Following the argument for the proof of Thm.~\ref{thm:rgrad_det}, the first three terms can be similarly bounded as follows
\begin{align*}
&I_4 \leq (8\vep_0+2\vep_\alpha(1+4\vep_0))\ln X_l-X\rn_F,\\
&I_5 \leq \frac{\vep_0}{1-4\vep_0}\ln X_l-X\rn_F\leq(1+\vep_\alpha)\vep_0\ln X_l-X\rn_F,\\
&I_6\leq(1+\vep_\alpha)\vep_0\ln X_l-X\rn_F.
\end{align*}

To bound $I_7$, first note that $\beta_l\P_{T_l}(P_{l-1})$ can be expressed in terms of all the previous gradients 
\begin{align*}
\beta_l\P_{T_l}(P_{l-1}) = \sum_{j=0}^{l-1}\lb\prod_{i=j+1}^l\beta_i\rb\lb\prod_{k=j}^{l-1}\P_{T_k}\rb(G_j).
\end{align*}
So
\begin{align*}
I_7&\leq 2|\alpha_l|\sum_{j=0}^{l-1}\lb\prod_{i=j+1}^l\beta_i\rb\ln\lb\prod_{k=j}^{l-1}\P_{T_k}\rb(G_j)\rn_F\\
&\leq 2|\alpha_l|\sum_{j=0}^{l-1}\vep_\beta^{l-j}\ln \P_{T_j}(G_j)\rn_F\\
&=2|\alpha_l|\sum_{j=0}^{l-1}\vep_\beta^{l-j}\ln \P_{T_j}\P_\Omega(X_j-X)\rn_F\\
&\leq 2|\alpha_l|\sum_{j=0}^{l-1}\vep_\beta^{l-j}\lb\ln \P_{T_j}\P_\Omega\P_{T_j}(X_j-X)\rn_F+\ln  \P_{T_j}\P_\Omega(\I-\P_{T_j})(X_j-X)\rn_F\rb\\
&\leq 2|\alpha_l|\sum_{j=0}^{l-1}\vep_\beta^{l-j}\lb \ln \P_{T_j}\P_\Omega\P_{T_j}(X_j-X)\rn_F+\ln\P_\Omega\P_{T_j}\rn\frac{\ln X_j-X\rn_F^2}{\sigmin{X}}\rb\\
&\leq \frac{1+\vep_\alpha}{p}\sum_{j=0}^{l-1}\vep_\beta^{l-j}\lb 2(1+4\vep_0)p\ln X_j-X\rn_F+p\vep_0\ln X_j-X\rn_F\rb\\
&=(1+\vep_\alpha)\sum_{j=0}^{l-1}\vep_\beta^{l-j}\lb 2(1+4\vep_0)+\vep_0\rb\ln X_j-X\rn_F,
\end{align*}
where the second inequality follows from the bound for $\beta_i$ in Lem.~\ref{lem:rcg_alpha_beta} and the fact 
\begin{align*}
{\ln\lb\prod_{k=j}^{l-1}\P_{T_k}\rb(G_j)\rn_F\leq \ln \P_{T_j}(G_j)\rn_F,}
\end{align*}
the fourth inequality follows from Lem.~\ref{lem:tech1}, and the last inequality follows from the bound for $\alpha_l$
in Lem.~\ref{lem:rcg_alpha_beta}, \eqref{eq:rcg_extra_assump},  \eqref{eq:rcg_popt_bd}, and \eqref{eq:rcg_iso1}.

Inserting the bounds for $I_4,~I_5,~I_6$ and $I_7$ into \eqref{eq:rcg_recursive} gives
\begin{align*}
\ln X_{l+1}-X\rn_F&\leq \lb 10\vep_0+2\vep_\alpha(1+5\vep_0)\rb\ln X_l-X\rn_F+(1+\vep_\alpha) (2(1+4\vep_0)+\vep_0)\sum_{j=0}^{l-1}\vep_\beta^{l-j}\ln X_j-X\rn_F\\
&\leq \lb 10\vep_0+2\vep_\alpha(1+5\vep_0)\rb\ln X_l-X\rn_F+(2+10\vep_0+2\vep_\alpha(1+5\vep_0))\sum_{j=0}^{l-1}\vep_\beta^{l-j}\ln X_j-X\rn_F\\
&:=\rho_1\ln X_l-X\rn_F+\rho_2\sum_{j=0}^{l-1}\vep_\beta^{l-j}\ln X_j-X\rn_F
\end{align*}
Define 
\begin{align*}
&\tau_1 = \rho_1+\vep_\beta=\frac{18\vep_0-10\vep_0\kappa_1(1+4\vep_0)}{(1-4\vep_0)-\kappa_1(1+4\vep_0)}+\frac{4\kappa_2\vep_0+\kappa_1\kappa_2}{1-4\vep_0},\\
&\tau_2 = (\rho_2-\rho_1)\vep_\beta=\frac{8\kappa_2\vep_0+2\kappa_1\kappa_2}{1-4\vep_0},\\
&\nu_{cg} = \frac{1}{2}\lb\tau_1+\sqrt{\tau_1^2+\tau_2}\rb.
\end{align*}
When $l=0$, Alg.~\ref{alg:rcg} is the same as Alg.~\ref{alg:rgrad}, so
\begin{align*}
\ln X_1-X\rn_F&\leq\frac{18\vep_0}{1-4\vep_0}\ln X_0-X\rn_F\leq \nu_{cg}\ln X_0-X\rn_F,
\end{align*}
following from \eqref{eq:rgrad_contract}.
Therefore Lem.~\ref{lem:tech4} implies 
if $\tau_1+\tau_2<1$, we have $\nu_{cg}<1$ and 
\begin{align*}
\ln X_{l+1}-X\rn_F\leq \nu_{cg}^{l+1}\ln X_0-X\rn_F.
\end{align*}
Similarly, \eqref{eq:rcg_extra_assump} is only required at $l=0$ since we have 
a contractive sequence. 
\end{proof}
\subsection{Proof of Lemma~\ref{lem:error_of_init1}}\label{subsec:proof_err_init1}
Let $W_0=p^{-1}\P_\Omega(X)$. First we have 
\begin{align*}
\ln X_0-X\rn&\leq \ln X_0-W_0\rn+\ln W_0-X\rn\leq 2\ln W_0-X\rn\leq 2\sqrt{\frac{8\beta n^3\log(n)}{3m}}\ln X\rn_\infty,
\end{align*}
where the last inequality holds true with probability at least $1-2n^{1-\beta}$ following from Thm.~\ref{thm:operator_norm}.
Consequently,
\begin{align*}
\ln X_0-X\rn_F&\leq \sqrt{2r}\ln X_0-X\rn\leq \sqrt{\frac{64\beta n^3r\log(n)}{3m}}\ln X\rn_\infty\leq \sqrt{\frac{64\beta\mu_1^2nr^2\log(n)}{3m}}\ln X\rn,
\end{align*}
which completes the proof.
\subsection{Proof of Lemma~\ref{thm:init2}}\label{subsec:proof_err_init2}
We first present a lemma which says the matrix returned by Alg.~\ref{alg:trim} is incoherent 
and its approximation error can be bounded by the approximation error of the input matrix.
\begin{lemma}\label{lem:init2_aux}
Let $Z_l=U_l\Sigma_lV_l^*$ be a rank $r$ matrix such that 
\begin{align*}
\ln Z_l-X\rn_F\leq\frac{\sigmin{X}}{10\sqrt{2}}.
\end{align*}
Then the matrix $\widehat{Z}_l$ returned by Alg.~\ref{alg:trim} satisfies 
\begin{align*}
\ln\P_{\widehat{U}_l}\lb e_i\rb\rn\leq\frac{10}{9}\sqrt{\frac{\mu_0 r}{n}}\mbox{ and }\ln\P_{\widehat{V}_l}\lb e_j\rb\rn\leq\frac{10}{9}\sqrt{\frac{\mu_0 r}{n}}
\end{align*}
for $1\leq i, j\leq n$, where the columns of $\widehat{U}_l$ and $\widehat{V}_l$ are  the right and left singular vectors of $\widehat{Z}_l$ respectively. Furthermore,
\begin{align*}
\ln\widehat{Z}_l-X\rn_F\leq 8\kappa\ln Z_l-X\rn_F.
\end{align*}
\end{lemma}
\begin{proof}
For simplicity, let $d=\ln Z_l-X\rn_F$. By Lem.~\ref{lem:tech1}, we have 
\begin{align*}
\ln U_lU_l^*-UU^*\rn_F\leq \frac{\sqrt{2}d}{\sigmin{X}}
\quad\mbox{and}\quad
\ln V_lV_l^*-VV^*\rn_F\leq \frac{\sqrt{2}d}{\sigmin{X}}.
\end{align*}
This together with Lem.~\ref{lem:tech5} implies that there exist two unitary matrices $Q_u\in\R^{r\times r}$ and $Q_v\in\R^{r\times r}$ such that 
\begin{align*}
\ln U_l-UQ_u\rn_F\leq\frac{\sqrt{2}d}{\sigmin{X}}
\quad\mbox{ and }\quad
\ln V_l-VQ_v\rn_F\leq\frac{\sqrt{2}d}{\sigmin{X}}.
\end{align*}
It follows that 
\begin{align*}
\ln\Sigma_l-Q_u^*\Sigma Q_v\rn_F&=\ln U_l^*Z_lV_l-(UQ_u)^*X(VQ_v)\rn_F\\
&\leq \ln U_l^*Z_lV_l-(UQ_u)^*Z_lV_l\rn_F+\ln (UQ_u)^*Z_lV_l-(UQ_u)^*XV_l\rn_F\\
&\quad + \ln (UQ_u)^*XV_l-(UQ_u)^*X(VQ_v)\rn_F\\
&\leq\ln U_l-UQ_u\rn_F\ln Z_l\rn+\ln Z_l-X\rn_F+\ln X\rn\ln V_l-VQ_v\rn_F\\
&\leq \frac{\sqrt{2}d(d+\sigmax{X})}{\sigmin{X}}+d+\frac{\sqrt{2}d\sigmax{X}}{\sigmin{X}}\\
&\leq 4\kappa d,\numberthis\label{eq:init2_S}
\end{align*}
where $\kappa$ is the condition number of $X$, and we have used the assumption $d\leq\sigmin{X}/10\sqrt{2}\leq \sigmax{X}/10\sqrt{2}$ and the fact
\begin{align*}
\ln Z_l\rn\leq \ln X\rn+\ln Z_l-X\rn\leq \sigmax{X}+d.
\end{align*}
in the last two inequalities. 

Recall that $A_l$ and $B_l$ in Alg.~\ref{alg:trim} are defined as 
\begin{align*}
A_l^{(i)} = \frac{U_l^{(i)}}{\ln U_l^{(i)}\rn}\min\lb \ln U_l^{(i)}\rn,\sqrt{\frac{\mu_0 r}{n}}\rb,\quad\mbox{and}\quad B_l^{(i)} = \frac{V_l^{(i)}}{\ln V_l^{(i)}\rn}\min\lb \ln V_l^{(i)}\rn,\sqrt{\frac{\mu_0 r}{n}}\rb.
\end{align*}
{Because 
\begin{align*}
\ln (UQ_u)^{(i)}\rn\leq\sqrt{\frac{\mu_0r}{n}}\mbox{ and }\ln (VQ_v)^{(i)}\rn\leq\sqrt{\frac{\mu_0r}{n}},
\end{align*}
we have 
\begin{align*}
\ln A_l^{(i)}-(UQ_u)^{(i)}\rn\leq \ln U_l^{(i)}-(UQ_u)^{(i)}\rn
\quad\mbox{and}\quad
\ln B_l^{(i)}-(VQ_v)^{(i)}\rn\leq \ln V_l^{(i)}-(VQ_v)^{(i)}\rn.
\end{align*}}
Therefore
\begin{align*}
&\ln A_l-UQ_u\rn_F\leq \ln U_l-UQ_u\rn_F\leq\frac{\sqrt{2}d}{\sigmin{X}},\numberthis\label{eq:init2_U}\\
&\ln B_l-VQ_v\rn_F\leq \ln V_l-VQ_v\rn_F\leq\frac{\sqrt{2}d}{\sigmin{X}}.\numberthis\label{eq:init2_V}
\end{align*}
Since $\widehat{Z}_l=A_l\Sigma_lB_l^*$ by Alg.~\ref{alg:trim}, we have 
\begin{align*}
\ln \widehat{Z}_l-X\rn_F&=\ln A_l\Sigma_lB_l^*-(UQ_u)(Q_u^*\Sigma Q_v)(VQ_v)^*\rn_F\\
&\leq \ln A_l\Sigma_lB_l^*-(UQ_u)\Sigma_lB_l^*\rn_F + \ln (UQ_u)\Sigma_lB_l^*-(UQ_u)(Q_u^*\Sigma Q_v)B_l^*\rn_F\\
&\quad + \ln (UQ_u)(Q_u^*\Sigma Q_v)B_l^*-(UQ_u)(Q_u^*\Sigma Q_v)(VQ_v)^*\rn_F\\
&\leq \ln A_l-UQ_u\rn_F\ln\Sigma_l\rn\ln B_l\rn+\ln\Sigma_l-Q_u^*\Sigma Q_v\rn_F\ln B_l\rn+\ln\Sigma\rn\ln B_l-VQ_v\rn_F\\
&\leq \frac{\sqrt{2}d}{\sigmin{X}}(\sigmax{X}+d)\lb\frac{\sqrt{2}d}{\sigmin{X}}+1\rb + 4\kappa d\lb\frac{\sqrt{2}d}{\sigmin{X}}+1\rb+\sigmax{X}\frac{\sqrt{2}d}{\sigmin{X}}\\
&\leq 8\kappa d,
\end{align*}
where in the last two inequalities we use \eqref{eq:init2_S}, the assumption $d\leq\sigmin{X}/10\sqrt{2}\leq \sigmax{X}/10\sqrt{2}$  and the fact
\begin{align*}
\ln\Sigma_l\rn=\ln Z_l\rn\leq \ln X\rn+\ln Z_l-X\rn\leq \sigmax{X}+d.
\end{align*}

It remains to estimate the incoherence of $\widehat{Z}_l$.  Because $A_l$ and $B_l$
 are not necessarily orthogonal, we consider their QR factorizations:
 \begin{align*}
 A_l=\widetilde{U}_lR_u \mbox{ and }B_l=\widetilde{V}_lR_v. 
 \end{align*}
 First note that 
{ \begin{align*}
& \sigmin{A_l}\geq  1-\ln A_l-UQ_u\rn\geq 1-\frac{\sqrt{2}d}{\sigmin{X}}\geq \frac{9}{10},\\
&\sigmin{B_l}\geq 1-\ln B_l-VQ_v\rn\geq 1-\frac{\sqrt{2}d}{\sigmin{X}}\geq \frac{9}{10}
 \end{align*}
 following from \eqref{eq:init2_U}, \eqref{eq:init2_V}, the assumption $d\leq\sigmin{X}/10\sqrt{2}$ and the Weyl inequality. Therefore
 \begin{align*}
 \ln R_u^{-1}\rn\leq \frac{10}{9}\mbox{ and }\ln R_v^{-1}\rn\leq \frac{10}{9}.
 \end{align*}}
 Consequently,
 \begin{align*}
 \ln\widehat{U}_l^{(i)}\rn=\ln \widetilde{U}^{(i)}_l\rn=\ln A_l^{(i)}R_u^{-1}\rn\leq \frac{10}{9}\sqrt{\frac{\mu_0r}{n}}
\quad\mbox{ and }
 \ln\widehat{V}_l^{(i)}\rn=\ln \widetilde{V}^{(i)}_l\rn=\ln B_l^{(i)}R_v^{-1}\rn\leq \frac{10}{9}\sqrt{\frac{\mu_0r}{n}}
 \end{align*}
 following from the construction of $A_l$ and $B_l$.
\end{proof}

\begin{proof}[Proof of Lemma~\ref{thm:init2}]
First assume that 
\begin{align*}
\ln Z_l-X\rn_F\leq \frac{\sigmin{X}}{256\kappa^2}.\numberthis\label{eq:init2_bd_assump}
\end{align*}
It follows immediately from Lem.~\ref{lem:init2_aux} that $\hZ_l$ is  an incoherent matrix with  the incoherence parameter $\frac{100}{81}\mu_0$ and 
\begin{align*}
\ln\hZ_l-X\rn_F\leq 8\kappa\ln Z_l-X\rn_F.
\end{align*}

Analogous to \eqref{eq:rgrad_recursive}, the approximation error at the $(l+1)$th iteration can be decomposed 
as follows
\begin{align*}
\ln Z_{l+1}-X\rn_F&\leq 2\ln (\P_{\hT_l}-\frac{n^2}{\tilm}\P_{\hT_l}\P_{\Omega_{l+1}}\P_{\hT_l})(\hZ_l-X)\rn_F\\
&\quad+2\ln(\I-\P_{\hT_l})(\hZ_l-X)\rn_F\\
&\quad+2\ln\frac{n^2}{\tilm} \P_{\hT_l}\P_{\Omega_{l+1}}(\I-\P_{\hT_l})(\hZ_l-X)\rn_F\\
&:=I_8+I_9+I_{10}.\numberthis\label{eq:init2_recursive}
\end{align*}

Since $\hZ_l$ (and consequently $\hT_l$) is independent of $\Omega_{l+1}$,  Thm.~\ref{thm:isometry} implies
\begin{align*}
\ln \P_{\hT_l}-\frac{n^2}{\tilm}\P_{\hT_l}\P_{\Omega_{l+1}}\P_{\hT_l}\rn\leq\sqrt{\frac{3200\beta\mu_0nr\log(n)}{243\tilm}}.
\end{align*}
 with probability 
at least $1-2n^{2-2n}$. So
\begin{align*}
I_8&\leq 2\sqrt{\frac{3200\beta\mu_0nr\log(n)}{243\tilm}}\ln\hZ_l-X\rn_F\\
&\leq 16\kappa\sqrt{\frac{3200\beta\mu_0nr\log(n)}{243\tilm}}\ln Z_l-X\rn_F.
\end{align*}

For $I_9$, Lem.~\ref{lem:tech1} implies 
\begin{align*}
I_9\leq \frac{2\ln \hZ_l-X\rn_F^2}{\sigmin{X}}\leq \frac{128\kappa^2\ln Z_l-X\rn_F^2}{\sigmin{X}}\leq\frac{1}{2}\ln Z_l-X\rn_F,
\end{align*}
where the last inequality follows from \eqref{eq:init2_bd_assump}.

To bound $I_{10}$, note again $\hZ_l$ (and consequently $\hT_l$) is independent of $\Omega_{l+1}$ with the incoherence 
parameter $\frac{100}{81}\mu_0$. So Lem.~\ref{lem:tech3} implies 
\begin{align*}
&\ln\frac{n^2}{\tilm}\P_{T_l}\P_{\Omega_{l+1}}(\P_U-\P_{U_l})-\P_{T_l}(\P_U-\P_{U_l})\rn\leq\sqrt{\frac{4800\beta\mu_0nr\log(n)}{81\tilm}}.
\end{align*}
with probability at least $1-2n^{2-2\beta}$.
Since 
\begin{align*}
(\I-\P_{\hT_l})(\hZ_l-X)&=-(\I-\P_{\hT_l})(X)\\
&=-UU^*X+\hU_l\hU_l^*X+UU^*X\hV_l\hV_l^*-\hU_l\hU_l^*X\hV_l\hV_l^*\\
&=-(UU^*-\hU_l\hU_l^*)X(I-\hV_l\hV_l^*)\\
&=(UU^*-\hU_l\hU_l^*)(\hZ_l-X)(I-\hV_l\hV_l^*)\\
&=(\P_U-\P_{\hU_l})(\I-\P_{\hV_l})(\hZ_l-X),
\end{align*}
we have 
\begin{align*}
I_{10}&=2\ln \frac{n^2}{\tilm} \P_{\hT_l}\P_{\Omega_{l+1}}(\P_U-\P_{\hU_l})(\I-\P_{\hV_l})(\hZ_l-X)\rn_F\\
&=2\ln \frac{n^2}{\tilm} \P_{\hT_l}\P_{\Omega_{l+1}}(\P_U-\P_{\hU_l})(\I-\P_{\hV_l})(\hZ_l-X)-\P_{\hT_l}(\P_U-\P_{\hU_l})(\I-\P_{\hV_l})(\hZ_l-X)\vphantom{\frac{n^2}{\tilm}}\rn_F\\
&\leq2\ln\frac{n^2}{\tilm}\P_{T_l}\P_{\Omega_{l+1}}(\P_U-\P_{U_l})-\P_{T_l}(\P_U-\P_{U_l})\rn\ln \hZ_l-X\rn_F\\
&\leq 2\sqrt{\frac{4800\beta\mu_0 nr\log(n)}{81\tilm}}\ln \hZ_l-X\rn_F\\
&\leq 16\kappa\sqrt{\frac{4800\beta\mu_0 nr\log(n)}{81\tilm}}\ln Z_l-X\rn_F.
\end{align*}
where the second equality follows from 
\begin{align*}
&\P_{\hT_l}(\P_U-\P_{\hU_l})(\I-\P_{\hV_l})(\hZ_l-X)=\P_{\hT_l}(\I-\P_{\hT_l})(\hZ_l-X)=0.
\end{align*}

Putting the bounds for $I_8$, $I_9$ and $I_{10}$ together gives 
\begin{align*}
\ln Z_{l+1}-X\rn_F&\leq \lb\frac{1}{2}+182\kappa\sqrt{\frac{\beta \mu_0nr\log(n)}{\tilm}}\rb\ln Z_l-X\rn_F\\
&\leq\frac{5}{6}\ln Z_l-X\rn_F
\end{align*}
with probability at least $1-4n^{2-2\beta}$ provided 
\begin{align*}
\tilm\geq C\beta\mu_0\kappa^2nr\log(n)\numberthis\label{eq:mhat_l}
\end{align*}
with $C$ being sufficiently large.  Clearly \eqref{eq:init2_bd_assump} is also valid for $(l+1)$th iteration.

Since $Z_0 = \H_r\lb \frac{\tilm}{n^2}\P_{\Omega_0}(X)\rb$, \eqref{eq:init2_bd_assump} is valid for $l=0$ with probability $1-2n^{1-\beta}$ provided 
\begin{align*}
\tilm\geq C\beta \mu_1^2\kappa^6nr^2\log(n)\numberthis\label{eq:mhat_0}
\end{align*}
following from Lem.~\ref{lem:error_of_init1}. Therefore taking a maximum of the right hand sides of \eqref{eq:mhat_l} and \eqref{eq:mhat_0} gives 
\begin{align*}
\ln Z_L-X\rn_F\leq\lb\frac{5}{6}\rb^{L}\frac{\sigmin{X}}{256\kappa^2}
\end{align*}
with probability at least $1-2n^{1-\beta}-4Ln^{2-2\beta}$ provided
\begin{align*}
\tilm\geq C\beta\max\lcb\mu_0,\mu_1^2\rcb\kappa^6nr^2\log(n).
\end{align*}
\end{proof}

\section{Conclusion and Future Direction}\label{sec:conclusion}
This manuscript presents the sampling complexity for a class of  Riemannian gradient descent and conjugate gradient descent algorithms by interpreting them as iterative hard thresholding algorithms with subspace projections. To the best of our knowledge, this is the first work that provides  recovery guarantees of the Riemannian optimization algorithms on the embedded low rank matrix manifold for matrix completion. Our first sampling complexity for the algorithms with one step hard thresholding initialization is  $O(n^{1.5}r\log^{1.5}(n))$. We would like to know whether this result can be improved to the nearly optimal one $O(nr\hspace{0.05cm}\mbox{polylog}(n))$ \cite{candestao2009mc} since the simulation results in Sec.~\ref{subsec:phase} strongly suggests so. The extra $r$ term in the second recovery guarantee result comes from bounding the matrix spectral norm by the Frobenius norm. To further optimize the second sampling complexity, we may need to establish convergence of the algorithms in terms of the matrix spectral norm rather than the Frobenius norm.

It would be of interest to investigate whether we can extend our results to the case whether the unknown matrix is approximately low rank. Since in this case $\P_\Omega(X)$ can be decomposed as $\P_\Omega(X)=\P_\Omega(X_r)+\P_\Omega(X-X_r)$, where $X_r$ is the best rank $r$ approximation of $X$, it is equivalent to study the robustness of the algorithms under  additive noise which is interesting by itself. The empirical observations in Sec.~\ref{subsec:noise} suggests that both the Riemannian gradient descent and conjugate gradient descent algorithms are very robust under additive Gaussian white noise. Robustness analysis of the algorithms is scope for future work. It would also be desirable to consider more general  and practical sampling models in matrix completion for the Riemannian optimization algorithms (Algs.~\ref{alg:rgrad} and ~\ref{alg:rcg}).

The Riemannian gradient descent and conjugate gradient descent algorithms 
presented in this manuscript apply equally to other low rank reconstruction problems, such as phase retrieval \cite{GeSaPhase,phaselift1,wirtinger,kacz_phase,chencandes,wright_phase} and blind deconvolution \cite{blind_conv,strohmer_demixing}. This line of research will be pursued
independently in the future. In phase retrieval and blind deconvolution,  the underlying matrix after lifting is rank one. Since  the condition number of a rank one matrix is alway equal to one,  it is worth investigating whether we can obtain recovery guarantees of the form $O(n\hspace{0.05cm}\mbox{polylog}(n))$ with a universal constant. 
\section*{Acknowledgments}
KW has been supported by DTRA-NSF grant No.1322393. SL was supported in part by the Hong Kong RGC grant 16303114. KW would like to thank Xiaodong Li, Shuyang Ling and Thomas Strohmer for stimulating discussions about phase retrieval and blind deconvolution.
\bibliography{iht}
\bibliographystyle{plain}
\appendix
\section{Supplementary Results}
In this section, we will list the tail probability bounds  from the literature which have been 
used in this manuscript. We start with the Noncommutative Bernstein inequality. 
\begin{theorem}[Noncommutative Bernstein Inequality, \cite{AhlWin,recht2011simple}]\label{thm:nonc_bernstein} Let $Z_1,\cdots,Z_L$ be $n\times n$ independent zero-mean random matrices. Suppose $\sigma_k^2=\max\lcb\ln\E\lsb Z_kZ_k^*\rsb\rn, \ln\E\lsb Z_k^*Z_k\rsb\rn\rcb$ and $\ln Z_k\rn\leq R$ almost surely for all $k$. Then for any $t>0$,
\begin{align*}
\Prob\lsb\ln\sum_{k=1}^L Z_k\rn>t\rsb\leq 2n\exp\lb \frac{-t^2/2}{\sum_{k=1}^L\sigma_k^2+Rt/3}\rb.
\end{align*}
Moreover, if $t\leq\frac{1}{R}\sum_{k=1}^L\sigma_k^2$, 
\begin{align*}
\Prob\lsb\ln\sum_{k=1}^L Z_k\rn>t\rsb\leq 2n\exp\lb \frac{-\frac{3}{8}t^2}{\sum_{k=1}^L\sigma_k^2}\rb.
\end{align*}
\end{theorem}
Two tail probability bounds for  $\P_\Omega$ can be obtained using the Noncommutative Bernstein inequality under the sampling with 
replacement model.
\begin{theorem}[\cite{recht2011simple,gross2011recoverlowrank,candesrecht2009mc}]\label{thm:isometry}
Suppose $\Omega$ with $|\Omega|=m$ is a set of indices sampled independently and uniformly with replacement. Let $X=U\Sigma V^*\in\R^{n\times n}$  be a rank $r$ matrix with the corresponding tangent space $T$. 
Assume 
\begin{align*}
\ln\P_U\lb e_i\rb\rn\leq\sqrt{\frac{\mu r}{n}}\quad\mbox{and}\quad\ln\P_V\lb e_j\rb\rn\leq\sqrt{\frac{\mu r}{n}}
\end{align*}
for $1\leq i, j\leq n$. 
Then for all $\beta >1$,
\begin{align*}
\ln \P_T-\frac{n^2}{m}\P_T\P_\Omega\P_T\rn\leq \sqrt{\frac{32\beta\mu nr\log(n)}{3m}}
\end{align*}
with probability at least $1-2n^{2-2\beta}$ provided $m\geq \frac{32}{3}\beta\mu nr\log(n)$.
\end{theorem}
\begin{theorem}[\cite{recht2011simple,gross2011recoverlowrank,candesrecht2009mc}]\label{thm:operator_norm}
Suppose $\Omega$ with $|\Omega|=m$ is a set of indices sampled independently and uniformly with replacement. Let $Z\in\R^{n\times n}$ be a fixed matrix. Then for all $\beta >1$,
\begin{align*}
\ln\lb\frac{n^2}{m}\P_\Omega-\I\rb (Z)\rn\leq \sqrt{\frac{8\beta n^3\log(n)}{3m}}\ln Z\rn_\infty
\end{align*}
with probability $1-2n^{1-\beta}$ provided $m\geq 6\beta n\log(n)$. 
\end{theorem}
\section{Proofs of Technical Lemmas}\label{sec:proofs_of_lemmas}
\subsection{Proof of Lemma~\ref{lem:tech1}}
The proofs for the first five inequalities in this lemma can be found in \cite{CGIHT_dense}. So it only 
remains to prove the last inequality. For any matrix $Z\in\R^{n\times n}$, we have 
\begin{align*}
\lb \P_{T_l}-\P_T\rb (Z) & = U_lU_l^*Z+ZV_lV_l^*-U_lU_l^*ZV_lV_l^*-UU^*Z-ZVV^*+UU^*ZVV^*\\
&=(U_lU_l^*-UU^*)Z(I-VV^*)+(I-U_lU_l^*)Z(V_lV_l^*-VV^*).
\end{align*}
Taking the Frobenius norm on both sides gives 
\begin{align*}
\ln \lb \P_{T_l}-\P_T\rb (Z)\rn_F & \leq \ln U_lU_l^*-UU^*\rn \ln Z\rn_F\ln I-VV^*\rn+\ln I-U_lU_l^*\rn \ln Z\rn_F\ln V_lV_l^*-VV^*\rn\\
&\leq\frac{2\ln X_l-X\rn_F}{\sigmin{X}}\ln Z\rn_F.
\end{align*}
\subsection{Proof of Lemma~\ref{lem:tech2}}
For any $Z\in\R^{n\times n}$, we have 
\begin{align*}
\ln\P_\Omega\P_T(Z)\rn_F^2&=\la \P_\Omega\P_T(Z), \P_\Omega\P_T(Z)\ra\\
&\leq{\frac{8}{3}\beta\log(n)\la\P_T(Z),\P_\Omega\P_T(Z)\ra}\\
&=\frac{8}{3}\beta\log(n)\la\P_T(Z),\P_T\P_\Omega\P_T(Z)\ra\\
&\leq\frac{8}{3}\beta\log(n) (1+\vep_0)p\ln\P_T(Z)\rn_F^2,
\end{align*}
where the first inequality follows from the first assumption and the second inequality follows from the 
second assumption.  So we have  $\ln\P_\Omega\P_T\rn\leq \sqrt{\frac{8}{3}\beta\log(n) (1+\vep_0)p}$ and
\begin{align*}
\ln\P_\Omega\P_{T_l}\rn&\leq \ln \P_\Omega(\P_{T_l}-\P_T)\rn+\ln\P_\Omega\P_T\rn\\
&\leq\Pom\frac{2\ln X_l-X\rn_F}{\sigmin{X}}+\ln\P_\Omega\P_T\rn\\
&\leq \frac{p^{1/2}\vep_0}{1+\vep_0}+\sqrt{\frac{8}{3}\beta\log(n) (1+\vep_0)p}\\
&\leq {\Pom(1+\vep_0)p^{1/2}},
\end{align*}
where the second inequality follows from \eqref{eq:tech1}. To prove \eqref{eq:tech2_eq2}, we use
\begin{align*}
\ln\P_{T_l}-p^{-1}\P_{T_l}\P_\Omega\P_{T_l}\rn&\leq\ln\P_{T_l}-\P_T\rn+p^{-1}\ln\P_{T_l}\P_\Omega\P_{T_l}-\P_{T_l}\P_\Omega\P_T\rn\\
&\quad+p^{-1}\ln \P_{T_l}\P_\Omega\P_T-\P_T\P_\Omega\P_T\rn+\ln \P_T-p^{-1}\P_T\P_\Omega\P_T\rn\\
&\leq\ln\P_{T_l}-\P_T\rn+p^{-1}\ln\P_\Omega\P_{T_l}\rn\ln\P_{T_l}-\P_T\rn\\
&\quad+p^{-1}\ln\P_\Omega\P_{T}\rn\ln\P_{T_l}-\P_T\rn+\ln \P_T-p^{-1}\P_T\P_\Omega\P_T\rn\\
&\leq \frac{2\ln X_l-X\rn_F}{\sigmin{X}}+p^{-1}\ln\P_\Omega\P_{T_l}\rn\frac{2\ln X_l-X\rn_F}{\sigmin{X}}\\
&\quad +p^{-1}\ln\P_\Omega\P_{T}\rn\frac{2\ln X_l-X\rn_F}{\sigmin{X}}+\ln \P_T-p^{-1}\P_T\P_\Omega\P_T\rn\\
&\leq 4\vep_0,
\end{align*}
where in the second inequality we utilize the fact ${\P_\Omega^*=\P_\Omega}$ so that $\ln\P_{T_l}\P_\Omega\rn=\ln\P_\Omega\P_{T_l}\rn$; and the last inequality follows from the assumption and the bounds for $\ln\P_\Omega\P_{T_l}\rn$ and $\ln\P_\Omega\P_{T}\rn$.
\subsection{Proof of Lemma~\ref{lem:tech3}}
Lemma~\ref{lem:tech3} is an asymmetric version of Thm.~\ref{thm:isometry}. We will use the Noncommutative Bernstein inequality (Thm.~\ref{thm:nonc_bernstein}) to prove it.  Let $\Omega = \lcb (i_k,j_k)\rcb_{k=1}^m$ be the sampled set of indices.
For any $Z\in\R^{n\times n}$, since 
\begin{align*}
\lb \P_U-\P_{U_l}\rb(Z)& = \sum_{i,j=1}^n\la \lb \P_U-\P_{U_l}\rb(Z), e_ie_j^*\ra e_ie_j^*=\sum_{i,j=1}^n\la Z, \lb \P_U-\P_{U_l}\rb(e_ie_j^*)\ra e_ie_j^*,
\end{align*}
we have 
\begin{align*}
\P_\Omega\lb \P_U-\P_{U_l}\rb(Z) = \sum_{k=1}^m\la Z, \lb \P_U-\P_{U_l}\rb(e_{i_k}e_{j_k}^*)\ra e_{i_k}e_{j_k}^*
\end{align*}
and 
\begin{align*}
&\P_{T_l}\P_\Omega\lb \P_U-\P_{U_l}\rb(Z) = \sum_{k=1}^m\la Z, \lb \P_U-\P_{U_l}\rb(e_{i_k}e_{j_k}^*)\ra \P_{T_l}\lb e_{i_k}e_{j_k}^*\rb.
\end{align*}
Let $\T_{i_k,j_k}:\R^{n\times n}\rightarrow \R^{n\times n}$ be a rank one linear operator defined as 
\begin{align*}
\T_{i_k,j_k} = \P_{T_l}\lb e_{i_k}e_{j_k}^*\rb\otimes \lb \P_U-\P_{U_l}\rb(e_{i_k}e_{j_k}^*).
\end{align*}
Then we have 
{\begin{align*}
&\P_{T_l}\P_\Omega\lb \P_U-\P_{U_l}\rb=\sum_{k=1}^m\T_{i_k,j_k},\numberthis\label{eq:op_sum}\\
&\T_{i_k,j_k}^*= \lb \P_U-\P_{U_l}\rb(e_{i_k}e_{j_k}^*)\otimes\P_{T_l}\lb e_{i_k}e_{j_k}^*\rb,\numberthis\label{eq:op_trans}\\
&\E\lsb \T_{i_k,j_k}\rsb=\frac{1}{n^2}\P_{T_l}\lb \P_U-\P_{U_l}\rb,\numberthis\label{eq:op_exp}\\
&\E\lsb \T_{i_k,j_k}^*\rsb=\frac{1}{n^2}\lb \P_U-\P_{U_l}\rb\P_{T_l}.\numberthis\label{eq:op_trans_exp}
\end{align*}}
Moreover, 
\begin{align*}
\ln \T_{i_k,j_k}\rn&\leq \ln \lb \P_U-\P_{U_l}\rb(e_{i_k}e_{j_k}^*)\rn_F\ln \P_{T_l} \lb e_{i_k}e_{j_k}^*\rb\rn_F\\
&\leq\lb \ln\P_U(e_{i_k}e_{j_k}^*)\rn_F+\ln\P_{U_l}(e_{i_k}e_{j_k}^*)\rn_F\rb\ln \P_{T_l} \lb e_{i_k}e_{j_k}^*\rb\rn_F\\
&\leq\frac{4\mu r}{n},
\end{align*}
where the last inequality follows from
\begin{align*}
&\ln \P_U(e_{i_k}e_{j_k}^*)\rn_F = \ln\P_U(e_{i_k})e_{j_k}^*\rn_F\leq \sqrt\frac{\mu r}{n},\numberthis\label{eq:op_up_pu}\\
&\ln \P_{U_l}(e_{i_k}e_{j_k}^*)\rn_F = \ln\P_{U_l}(e_{i_k})e_{j_k}^*\rn_F\leq \sqrt\frac{\mu r}{n},\numberthis\label{eq:op_up_pul}
\end{align*}
and 
\begin{align*}
\ln \P_{T_l} \lb e_{i_k}e_{j_k}^*\rb\rn_F^2&=\la  \P_{T_l} \lb e_{i_k}e_{j_k}^*\rb, e_{i_k}e_{j_k}^*\ra\\
&=\la \P_{U_l}(e_{i_k})e_{j_k}^*+e_{i_k}(\P_{V_l}(e_{j_k}))^*-\P_{U_l}(e_{i_k})(\P_{V_l}(e_{j_k}))^*,e_{i_k}e_{j_k}^*\ra\\
&=\ln \P_{U_l}(e_{i_k})\rn^2+\ln \P_{V_l}(e_{j_k})\rn^2-\ln \P_{U_l}(e_{i_k})\rn^2\ln \P_{V_l}(e_{j_k})\rn^2\\
&\leq \frac{2\mu r}{n}.\numberthis\label{eq:bd_of_pte}
\end{align*}
So 
\begin{align*}
\ln \T_{i_k,j_k}-\E\lsb \T_{i_k,j_k}\rsb\rn&\leq \ln \T_{i_k,j_k}\rn+\ln \E\lsb \T_{i_k,j_k}\rsb\rn\leq \frac{4\mu r}{n}+\frac{2}{n^2}\leq \frac{5\mu r}{n}.\numberthis\label{eq:op_up_bd}
\end{align*}
To apply the Bernstein inequality, we also need to bound
\begin{align*}
\ln\E\lsb\lb\T_{i_k,j_k}-\E(\T_{i_k,j_k})\rb^*\lb\T_{i_k,j_k}-\E(\T_{i_k,j_k})\rb\rsb\rn
\end{align*}
and 
\begin{align*}
\ln\E\lsb\lb\T_{i_k,j_k}-\E(\T_{i_k,j_k})\rb\lb\T_{i_k,j_k}-\E(\T_{i_k,j_k})\rb^*\rsb\rn.
\end{align*}
The first one can be proceeded as follows
\begin{align*}
&\ln\E\lsb\lb\T_{i_k,j_k}-\E(\T_{i_k,j_k})\rb^*\lb\T_{i_k,j_k}-\E(\T_{i_k,j_k})\rb\rsb\rn\\
&=\ln\E\lsb \T_{i_k,j_k}^*\T_{i_k,j_k}\rsb-\E\lsb\T_{i_k,j_k}^*\rsb\E\lsb \T_{i_k,j_k}\rsb\rn\\
&\leq\ln \E\lsb \T_{i_k,j_k}^*\T_{i_k,j_k}\rsb\rn+\ln\frac{1}{n^4}\lb \P_U-\P_{U_l}\rb\P_{T_l}\lb \P_U-\P_{U_l}\rb\rn\\
&\leq \ln\E\lsb \ln \P_{T_l}\lb e_{i_k}e_{j_k}^*\rb\rn_F^2\lb \P_U-\P_{U_l}\rb(e_{i_k}e_{j_k}^*)\otimes\lb \P_U-\P_{U_l}\rb(e_{i_k}e_{j_k}^*)\vphantom{\ln \P_{T_l}\lb e_{i_k}e_{j_k}^*\rb\rn_F^2}\rsb\vphantom{\ln \P_{T_l}\lb e_{i_k}e_{j_k}^*\rb\rn_F^2}\rn+\frac{4}{n^4}\\
&\leq\frac{2\mu r}{n}\ln\E\lsb\lb\P_U-\P_{U_l}\rb(e_{i_k}e_{j_k}^*)\otimes \lb\P_U-\P_{U_l}\rb(e_{i_k}e_{j_k}^*)\rsb\rn+\frac{4}{n^4}\\
&\leq\frac{2\mu r}{n}\frac{1}{n^2}\ln (\P_U-\P_{U_l})^2\rn+\frac{4}{n^4}\\
&\leq\frac{9\mu r}{n^3},\numberthis\label{eq:adj1_bd}
\end{align*}
where the first inequality follows from \eqref{eq:op_exp} and \eqref{eq:op_trans_exp}, the third inequality follows from 
\eqref{eq:bd_of_pte}, and the fourth inequality follows from the the fact
\begin{align*}
&\E\lsb\lb\P_U-\P_{U_l}\rb(e_{i_k}e_{j_k}^*)\otimes \lb\P_U-\P_{U_l}\rb(e_{i_k}e_{j_k}^*)\rsb= \frac{1}{n^2}\lb \P_U-\P_{U_l}\rb^2.
\end{align*}
Similarly, we have 
\begin{align*}
&\ln\E\lsb\lb\T_{i_k,j_k}-\E(\T_{i_k,j_k})\rb\lb\T_{i_k,j_k}-\E(\T_{i_k,j_k})\rb^*\rsb\rn\\
&=\ln\E\lsb \T_{i_k,j_k}\T_{i_k,j_k}^*\rsb-\E\lsb\T_{i_k,j_k}\rsb\E\lsb\T^*_{i_k,j_k}\rsb\rn\\
&\leq\ln\E\lsb \T_{i_k,j_k}\T_{i_k,j_k}^*\rsb\rn+\frac{1}{n^4}\ln\P_{T_l}(\P_U-\P_{U_l})^2\P_{T_l}\rn\\
&\leq \ln\E\left[\ln  \lb\P_U-\P_{U_l}\rb(e_{i_k}e_{j_k}^*)\right\|_F^2\P_{T_l}\lb e_{i_k}e_{j_k}^*\rb\otimes\P_{T_l}\lb e_{i_k}e_{j_k}^*\rb\vphantom{\ln  \lb\P_U-\P_{U_l}\rb(e_{i_k}e_{j_k}^*)\right\|_F^2}\right]\vphantom{\ln  \lb\P_U-\P_{U_l}\rb(e_{i_k}e_{j_k}^*)\right\|_F^2}\rn+\frac{4}{n^2}\\
&\leq\frac{4\mu r}{n}\ln\E\lsb \P_{T_l}\lb e_{i_k}e_{j_k}^*\rb\otimes\P_{T_l}\lb e_{i_k}e_{j_k}^*\rb\rsb\rn+\frac{4}{n^4}\\
&\leq\frac{4\mu r}{n}\frac{1}{n^2}\ln\P_{T_l}\rn+\frac{4}{n^4}\\
&\leq\frac{5\mu r}{n^3},\numberthis\label{eq:adj2_bd}
\end{align*}
where the first inequality follows from \eqref{eq:op_exp} and \eqref{eq:op_trans_exp}, the third inequality follows from 
\eqref{eq:op_up_pu} and \eqref{eq:op_up_pul}, and the fourth inequality follows from the fact
\begin{align*}
\E\lsb \P_{T_l}\lb e_{i_k}e_{j_k}^*\rb\otimes\P_{T_l}\lb e_{i_k}e_{j_k}^*\rb\rsb=\frac{1}{n^2}\P_{T_l}.
\end{align*}
Finally, the theorem follows by applying Thm.~\ref{thm:nonc_bernstein} to \eqref{eq:op_sum} with the bounds
\eqref{eq:op_up_bd}, \eqref{eq:adj1_bd} and \eqref{eq:adj2_bd}.
\subsection{Proof of Lemma~\ref{lem:tech4}}
By the assumption, \eqref{eq:tech4_eq1} is valid for $l=0$.  Assume it is valid for $j\leq l$. Then 
\begin{align*}
c_{l+1} &\leq \rho_1c_{l}+\rho_2\sum_{j=0}^{l-1}\gamma^{l-j}c_j\\
&\leq \rho_1\nu^lc_0+\rho_2\sum_{j=0}^{l-1}\gamma^{l-j}\nu^jc_0\\
&=\nu^{l}c_0\lb \rho_1+\rho_2\sum_{j=0}^{l-1}\lb\frac{\gamma}{\nu}\rb^{l-j}\rb\\
&\leq \nu^{l}c_0\lb\rho_1+\frac{\rho_2\gamma/\nu}{1-\gamma/\nu}\rb\\
&= \nu^{l}c_0\lb\rho_1+\frac{\rho_2\gamma}{\nu-\gamma}\rb.
\end{align*}
So it suffices to show that $\rho_1+\frac{\rho_2\gamma}{\nu-\gamma}=\nu$. It is equivalent to show that 
\begin{align*}
\nu^2-(\rho_1+\gamma)\nu-(\rho_2-\rho_1)\gamma=0.
\end{align*} 
which is true since 
\begin{align*}
\nu=\frac{1}{2}\lb\tau_1+\sqrt{\tau_1^2+4\tau_2}\rb
\end{align*}
with $\tau_1=\rho_1+\gamma$ and $\tau_2=(\rho_2-\rho_1)\gamma$.
\subsection{Proof of Lemma~\ref{lem:tech5}}
Since 
\begin{align*}
&\ln U_l-UQ\rn_F^2 =  2r-2\la U_l,UQ\ra,\\
&\ln U_lU_l^*-UU^*\rn_F^2 = 2r-2\la U_lU_l^*,UU^*\ra,
\end{align*}
it suffices to show that there exist a $Q$ such that 
\begin{align*}
\la U_l,UQ\ra\geq \la U_lU_l^*,UU^*\ra.
\end{align*}
It is equivalent to show that 
\begin{align*}
\la U^*U_l,Q\ra\geq\la U^*U_l,U^*U_l\ra
\end{align*}
for some unitary matrix $Q\in\R^{r\times r}$. Let $U_l^*U_l=Q_1\Lambda Q_2^*$ be the singular value decomposition of $U_l^*U_l$. Then we have $\Lambda(i,i)\leq 1~(1\leq i\leq r)$, and we can choose $Q=Q_1Q_2^*$.
\end{document}